\documentclass[11pt]{amsart}

\usepackage{a4wide, amsmath, amsfonts, amssymb, mathrsfs, amsthm, breqn}
\usepackage[hyperfootnotes=false]{hyperref}
\usepackage[dvipsnames]{xcolor}
\usepackage{shuffle}
\allowdisplaybreaks[2]

\numberwithin{equation}{section}

\newtheorem{theorem}{Theorem}[section]
\newtheorem{lemma}[theorem]{Lemma}
\newtheorem{corollary}[theorem]{Corollary}
\newtheorem{remark}[theorem]{Remark}
\newtheorem{definition}[theorem]{Definition}

\newtheorem{proposition}[theorem]{Proposition}

\newtheorem{example}[theorem]{Example}

\newcommand{\dd}{\,\mathrm{d}}
\renewcommand{\d}{\mathrm{d}}

\renewcommand{\epsilon}{\varepsilon}
\renewcommand{\phi}{\varphi}
\newcommand{\R}{\mathbb{R}}
\newcommand{\N}{\mathbb{N}}
\newcommand{\E}{\mathbb{E}}
\newcommand{\X}{\mathbb{X}}

\renewcommand{\P}{\mathbb{P}}
\newcommand{\bX}{\mathbf{X}}
\newcommand{\bY}{\mathbf{Y}}
\newcommand{\cA}{\mathcal{A}}
\newcommand{\cB}{\mathcal{B}}
\newcommand{\cF}{\mathcal{F}}
\newcommand{\cL}{\mathcal{L}}
\newcommand{\hX}{\widehat{X}}

\newcommand{\hbX}{\widehat{\mathbf{X}}}
\newcommand{\hbY}{\widehat{\mathbf{Y}}}
\newcommand{\hbW}{\widehat{\mathbf{W}}}
\newcommand{\bW}{\mathbf{W}}
\newcommand{\bbX}{\mathbb{X}}
\newcommand{\hbbX}{\widehat{\mathbb{X}}}
\newcommand{\hbbY}{\widehat{\mathbb{Y}}}
\newcommand{\hbbW}{\widehat{\mathbb{W}}}
\newcommand{\ver}[1]{{\vert\kern-0.25ex\vert\kern-0.25ex\vert #1 \vert\kern-0.25ex\vert\kern-0.25ex\vert}}

\DeclareMathOperator{\spn}{span}

\title[Global universal approximation with Brownian signatures]{Global universal approximation with Brownian signatures}

\author[Ceylan]{Mihriban Ceylan}
\address{Mihriban Ceylan, University of Mannheim, Germany}
\email{mihriban.ceylan@uni-mannheim.de}

\author[Pr{\"o}mel]{David~J. Pr{\"o}mel}
\address{David~J. Pr{\"o}mel, University of Mannheim, Germany}
\email{proemel@uni-mannheim.de}

\date{\today}

\begin{document}

\begin{abstract}
  We establish $L^p$-universal approximation theorems for general path-dependent and non-anticipative functionals on suitable rough path spaces, showing that linear functionals acting on signatures of time-extended rough paths are dense with respect to the $L^p$-distance. To that end, we derive global universal approximation theorems for weighted rough path spaces. We demonstrate that these $L^p$-universal approximation theorems apply to Gaussian processes, in particular, to fractional Brownian motion. As a consequence, linear functionals on the signature of the time-extended Brownian motion can approximate any $p$-integrable stochastic process adapted to the Brownian filtration, including solutions to stochastic differential equations.
\end{abstract}

\maketitle


% Classification
% --------------
% 60L10 - Rough analysis (Signatures and data streams)
% 91G99 - Actuarial science and mathematical finance (.. in this section)
% 60J65 - Stochastic processes (Brownian motion)
% 60H10 - Stochastic analysis (Stochastic ordinary differential equations)

%\tableofcontents

\section{Introduction}

The efficient approximation of functionals on path spaces is a key challenge in numerous areas, including machine learning, mathematical finance, and data-driven modeling of random dynamical systems. In recent years, so-called signature methods have emerged as a powerful framework for representing and approximating path-dependent functionals; see, for instance, \cite{McLeod2025,Bayer2025a}. The concept of signatures was introduced by K.-T. Chen \cite{Chen1954} in the 1950s and has since been extensively studied, most notably in the context of rough path theory \cite{Lyons2007}. Roughly speaking, the signature of a continuous path $X \colon [0,T] \to \mathbb{R}^d$ is the collection of its iterated integrals, which is known to faithfully represent the main characteristics of the path, see \cite{Hambly2010,Boedihardjo2016}.

At the heart of signature methods lie universal approximation theorems, which assert that continuous functionals on suitable path spaces can be approximated arbitrarily well on compact sets by linear functionals acting on signatures; see, for example, \cite{Levin2013,Kiraly2019,Lyons2020}. Owing to these approximation properties and their rich algebraic structure, signatures are often viewed as natural analogues of polynomials on path spaces. This viewpoint has led to a wide range of applications across disciplines. In machine learning and data science, signature methods have been successfully employed for tasks such as image and texture classification \cite{Graham2013}, the generation of synthetic data \cite{Kidger2019}, and topological data analysis \cite{Chevyrev2020}. In mathematical finance, signature methods have found numerous applications, including the pricing of path-dependent European and American options \cite{Lyons2019,Lyons2020,Bayraktar2024,Bayer2025,Bayer2025b}, model calibration \cite{Cuchiero2023,Cuchiero2025}, optimal execution \cite{Kalsi2020}, portfolio optimization \cite{Cuchiero2024b}, and stochastic optimal control \cite{Bank2025}.

While these signature-based universal approximation theorems are of considerable theoretical and practical interest, they are typically restricted to approximations on compact sets and to general path-dependent functionals. These limitations significantly reduce their applicability, in particular in mathematical finance and in the modeling of random dynamical systems. This issue is already apparent from the well-known fact that the sample paths of many fundamental stochastic processes used in financial modelling, such as Brownian motion, do not belong to any fixed compact subset of a path space with positive probability. Moreover, in decision-making problems under uncertainty --- such as optimal execution and portfolio selection --- relevant functionals are often path-dependent but necessarily non-anticipative, since decisions can only depend on the current and past of the underlying dynamics. These considerations have motivated the development of global universal approximation theorems for both general and non-anticipative functionals, formulated either in weighted function spaces or in $L^p$-spaces.

As a first contribution, we establish $L^p$-universal approximation theorems (Theorems~\ref{thm:Lpmain} and \ref{thm:Lpmain 2}) for both general path-dependent and non-anticipative functionals on suitable rough path spaces, formulated in terms of the classical signature. More precisely, these results show that linear functionals acting on the signatures of time-extended rough paths are dense with respect to the standard $L^p$-metric, provided the underlying finite measure satisfies an exponential moment condition. To prove these approximation results, we derive global universal approximation theorems (Propositions~\ref{prop: weighted UAT 1} and~\ref{prop: weighted UAT 2}) on suitably weighted spaces of (stopped) rough paths, relying on a weighted version of the Stone--Weierstrass theorem established in \cite{Cuchiero2024}. The concept of stopped rough paths used throughout follows the common rough path framework recently used in, e.g., \cite{Kalsi2020,Bayer2025,Cuchiero2025}, and can be considered as the natural analogue of stopped paths appearing in the context of functional It{\^o} calculus; see \cite{Cont2013,Dupire2019}.

The present work is related to recent advances on global universal approximation results for signatures. In contrast to the classical signature employed in the $L^p$-universal approximation theorems established in this paper, the results in \cite{Schell2023} and \cite{Bayer2025} are derived using so-called robust signatures, which were introduced in \cite{Chevyrev2022} as a normalized variant of the classical signature. Recall that the classical signature comes with numerical advantages like analytic formulas for expected signatures are available, whereas such tractability may be lost when working with the robust signature. Moreover, the approaches developed in \cite{Schell2023} and \cite{Bayer2025} differ substantially from the one pursued here; for a more detailed comparison, we refer to Remark~\ref{rem: UAT for robust signatures}. With regard to universal approximation theorems for weighted spaces, our analysis builds on a modification of the results in \cite{Cuchiero2024}, which we extend here to the setting of stopped rough paths. In contrast to \cite{Cuchiero2024}, where weakly geometric $\alpha$-H{\"o}lder rough paths are considered, we work with geometric $\alpha$-H{\"o}lder rough paths, which form a Polish space and are therefore more suitable for measure-theoretic arguments. A related weighted-space approximation result is obtained in \cite{Cuchiero2024b} for (Stratonovich-enhanced) stopped continuous semimartingales.

As a second contribution, we establish that the $L^p$-universal approximation theorems apply to suitable time-extended Gaussian processes, including (fractional) Brownian motion and their affine transformations. Consequently, linear functionals acting on the signatures of time-extended Brownian motion are shown to approximate any $p$-integrable stochastic process adapted to the Brownian filtration, including solutions to stochastic differential equations. The key technical ingredient is the verification of the required exponential moment condition under the corresponding Gaussian measure, in particular under the Wiener measure. We also refer to \cite{Schell2023,Bayer2025} for related approximation results based on robust signatures.

The $L^p$-universal approximation results developed in this paper are particularly well suited to applications in stochastic analysis and mathematical finance. They provide a rigorous theoretical foundation for signature-based models and methods in mathematical finance, including applications to the pricing and hedging of financial derivatives, model calibration, portfolio optimization, and stochastic control. In particular, our results justify the universality of signature-based models driven by Brownian noise, which have recently been introduced as flexible alternatives to classical stochastic differential equation models; see, e.g., \cite{Arribas2021,Cuchiero2023,Cuchiero2025}. Indeed, Proposition~\ref{prop: SDE signature approx} shows that such models can approximate the solutions of a broad class of stochastic differential equations, independently of the specific drift and diffusion coefficients. This establishes signature-based models as a universal approximation class for Brownian-driven stochastic dynamics, which are frequently used in financial modelling.

\medskip
\noindent\textbf{Organization of the paper:} In Section~\ref{sec: preliminaries}, we recall the underlying concepts of weighted spaces, signatures, and rough path theory. The universal approximation theorems in $L^p$ and weighted spaces are established in Section~\ref{sec: global approx}, both for general path-dependent and non-anticipative functionals on suitable rough path spaces. In Section~\ref{sec: gaussian process}, we demonstrate that these universal approximation results apply to Gaussian processes, like the fractional Brownian motion, and to $p$-integrable progressively measurable stochastic processes adapted to the Brownian filtration, including solutions to stochastic differential equations. Appendix~\ref{sec: appendix} presents auxiliary results on the space of stopped rough paths.

\medskip
\noindent\textbf{Acknowledgments:} M. Ceylan gratefully acknowledges financial support by the doctoral scholarship programme from the Avicenna-Studienwerk, Germany.

\section{Preliminaries}\label{sec: preliminaries}

In this section, we introduce the notation and essential background on weighted spaces, signatures, and rough path theory. We refer to \cite{Friz2010,Friz2020,Cuchiero2024} for a more detailed introduction to these topics.

\subsection{Essentials on weighted spaces}

Let $T>0$ be a fixed finite time horizon and, for $d\in\N$, let $\R^d$ be the standard $d$-dimensional Euclidean space equipped with the norm $|x|:=(\sum_{i=1}^d x_i^2)^{1/2}$ for $x=(x_1,\dots,x_d)\in\R^d$. The space of continuous linear maps $f$ from the normed space $(X,\|\,\cdot\,\|_X)$ to the normed space $(Y,\|\,\cdot\,\|_Y)$ is denoted by $\cL(X;Y)$, which is equipped with the norm $\|f\|_{\cL(X;Y)}:=\sup_{x\in X,\|x\|_X\le 1}\|f(x)\|_Y$. Furthermore, if $Y=\R$, the topological dual space of $X$, denoted by $X^\ast$, is identified with $\cL(X;\R)$. Elements of $X^\ast$ are linear functionals $\ell\colon X\to\R$ and the norm on $X^\ast$ is defined by $\|\ell\|_{X^\ast}:=\sup_{x\in X,\|x\|_X\le 1}|\ell(x)|$.

\medskip

For a Hausdorff topological space $(X,\tau_X)$ and a normed space $(E,\|\,\cdot\,\|_E)$, the space of continuous functions $f\colon X\to E$ is denoted by $C(X;E)$ and $C_b(X;E)\subseteq C(X;E)$ denotes the vector subspace of bounded functions. Whenever $E=\R$, we simplify the notation to $C(X):=C(X;\R)$ and $C_b(X):=C_b(X;\R)$, respectively. We write $C_b^k=C_b^k(\R^m;\cL(\R^d;\R^m))$ for the space of $k$-times continuously differentiable functions $f\colon \R^m\to\cL(\R^d;\R^m)$ such that $f$ and all its derivatives up to order $k$ are continuous and bounded, and equip the space $C_b^k=C_b^k(\R^m;\cL(\R^d;\R^m))$ with the norm
\begin{equation*}
  \|f\|_{C_b^k}:=\|f\|_{\infty}+\|Df\|_{\infty}+\ldots+\|D^kf\|_{\infty},
\end{equation*}
where $D^r f$ denotes the $r$-th order derivative of $f$ and $\|\,\cdot\,\|_\infty$ denotes the supremum norm on the corresponding spaces of operators.

For a measure space $(X,\mathcal A,\mu)$ and $1\le p<\infty$, the (vector-valued) Lebesgue space $L^p(X,\mu;\R^d)$ is defined as the space of (equivalence classes of) $\mathcal A$-measurable functions $f\colon X\to\R^d$ such that
\begin{equation*}
  \|f\|_{L^p(X,\mu;\R^d)} := \Bigl(\int_X |f(x)|^p \,\dd\mu(x)\Bigr)^{\frac{1}{p}} < \infty.
\end{equation*} 
For $d=1$, we simply write $L^p(X):=L^p(X,\mu):= L^p(X,\mu;\R)$ and $\|\cdot\|_{L^p(X)}:=\|\cdot\|_{L^p(X,\mu;\R)}$.

\medskip

In the following, we recall the framework of weighted spaces introduced in \cite{Cuchiero2024}, with slight adaptations that are crucial for our purposes. We begin by defining a weighted space and, subsequently, the corresponding weighted function space.

\medskip

Let $(X,\tau_X)$ be a completely regular Hausdorff topological space. A function $\psi\colon X\to (0,\infty)$ is called an admissible weight function if every pre-image $K_R:=\psi^{-1} ((0,R])=\{x\in X: \psi(x)\le R\}$ is compact with respect to $\tau_X$, for all $R>0$. In this case, we call the pair $(X,\psi)$ a weighted space.

Furthermore, we define the vector space 
\begin{equation*}
  B_\psi(X):=\Bigl\{f\colon X\to \R: \sup_{x\in X}\frac{|f(x)|}{\psi(x)}<\infty\Bigr\},
\end{equation*}
consisting of functions $f\colon X\to \R$, whose growth is controlled by the growth of the weight function $\psi\colon X\to (0,\infty)$, which we equip with the weighted norm $\|\,\cdot\,\|_{\cB_\psi(X)}$ given by
\begin{equation}\label{eq:weightednorm}
  \|f\|_{\cB_\psi(X)}:=\sup_{x\in X}\frac{|f(x)|}{\psi(x)},\quad f\in B_\psi(X).
\end{equation}
Note that the embedding $C_b(X)\hookrightarrow B_\psi(X)$ is continuous, allowing us to introduce the space
\begin{equation*}
  \mathcal{B}_\psi(X)
  := \overline{C_b(X)}^{\,\|\cdot\|_{\mathcal{B}_\psi(X)}},
\end{equation*}
which is the closure of $C_b(X)$ with respect to the norm $\|\,\cdot\,\|_{\cB_\psi(X)}$. Note that $\mathcal{B}_\psi(X)$ is a Banach space with the norm~\eqref{eq:weightednorm}. We refer to $\cB_\psi(X)$ as a weighted function space.

\subsection{Algebraic setting for signatures}

The $n$-fold tensor product of $ \R^d$ is given by
\begin{equation*}
  (\R^d)^{\otimes 0}:=\R \quad \text{and}\quad(\R^d)^{\otimes n}:=\underbrace{\R^d\otimes\ldots\otimes\R^d}_{n},\quad \text{for }n\in \N.
\end{equation*}
Let $(e_1, \ldots, e_d)$ be the canonical basis of $\R^d$. It is well-known that $\{e_{i_1}\otimes\cdots\otimes e_{i_n}: i_1,\ldots,i_n\in \{1,\ldots,d\}\}$ is a canonical basis for $(\R^d)^{\otimes n}$ and we denote by $e_\emptyset$ the basis element of $(\R^d)^{\otimes 0}$.
Then, every $a^{(n)}\in(\R^d)^{\otimes n}$ admits the coordinate representation
\begin{equation*}
a^{(n)}=\sum_{i_1,\dots,i_n=1}^d a_{i_1,\dots,i_n}\, e_{i_1}\otimes\cdots\otimes e_{i_n},
\end{equation*}
and we equip $(\R^d)^{\otimes n} $ with the usual Euclidean norm
\begin{equation*}
  | a^{(n)}|_{(\R^d)^{\otimes n}}:=\bigg(\sum_{i_1,\ldots,i_n=1}^d|a_{i_1,\ldots,i_n}|^2\bigg)^{1/2},\quad\text{for }  a^{(n)}\in (\R^d)^{\otimes n}.
\end{equation*}
When no confusion may arise, we write $| a^{(n)} |$ instead of $| a^{(n)}|_{(\R^d)^{\otimes n}}$.

For $d \in \mathbb{N}$, the extended tensor algebra on $\R^{d}$ is defined as
\begin{equation*}
  T((\R^{d})):=\Bigl\{\mathbf{a}:=(a^{(0)}, \ldots,  a^{(n)}, \ldots): a^{(n)} \in(\R^{d})^{\otimes n}\Bigr\},
\end{equation*}
and $a^{(i)}$ is called tensor of level $i$. We equip $T((\R^d))$ with the standard addition ``$+$'', tensor multiplication ``$\otimes$'', and scalar multiplication ``$\cdot$'' defined by
\begin{align*}
  \mathbf{a}+\mathbf{b} & :=\Bigl( a^{(0)}+ b^{(0)}, \ldots,  a^{(n)}+ b^{(n)}, \ldots\Bigr), \\
  \mathbf{a} \otimes \mathbf{b} & :=\Bigl( c^{(0)}, \ldots,  c^{(n)}, \ldots\Bigr),\\
  \lambda\cdot \mathbf{a} & :=\Bigl(\lambda  a^{(0)}, \ldots, \lambda a^{(n)}, \ldots\Bigr),
\end{align*}
for $\mathbf{a}=(a^{(n)})_{n=0}^{\infty}, \mathbf{b} =(b^{(n)})_{n=0}^{\infty}\in T((\R^{d}))$ and $\lambda \in \R$, where $ c^{(n)}:=\sum_{k=0}^{n} a^{(k)} \otimes  b^{(n-k)}$. Let us remark that $(T((\mathbb{R}^{d})),+, \cdot, \otimes)$ is a real non-commutative algebra with neutral element $\mathbf{1}=(1,0, \ldots, 0, \ldots)$. Similarly, we define the truncated tensor algebra of order $N \in \N$ by
\begin{equation*}
  T^{N}(\R^{d}):=\Bigl\{\mathbf{a} \in T((\R^{d})):  a^{(n)}=0, \forall n>N\Bigr\},
\end{equation*}
which we equip with the norm 
\begin{equation*}
  \|\mathbf a\|_{T^N(\R^d)}:=\max_{n=0,\ldots, N}| a^{(n)}|_{(\R^d)^{\otimes n}},\quad \text{for } \mathbf a=(a^{(n)})_{n=0}^N \in T^N(\R^d).
\end{equation*}
Note that $T^{N}(\R^{d})$ has dimension $\sum_{i=0}^{N} d^{i}$. Additionally, we define the tensor algebra $T(\R^d)=\bigcup_{n\in\N}T^n(\R^d)$ and for $c \in \R$ we define the subset $T_c((\R^d)) := \{\mathbf{a} = (a^{(n)})_{n=0}^\infty \in T((\R^d)): a^{(0)} = c\}$ and $T_c^N(\R^d):=\{\mathbf a\in T^N(\R^d): a^{(0)}=c\}$. Observe that $T_1^N(\R^d)$ is a Lie group under $\otimes$, with unit element $\mathbf 1=(1,0,\ldots,0)$. 

 The Lie algebra that is generated by $\{\mathbf{e}_1, \dots, \mathbf{e}_d\}$, where $\mathbf{e}_i := (0,e_i,0,\dots) \in T(\R^d)$, and the commutator bracket
\begin{equation*}
  [\mathbf{a},\mathbf{b}] = \mathbf{a} \otimes \mathbf{b} - \mathbf{b} \otimes \mathbf{a}, \qquad \mathbf{a}, \mathbf{b} \in T(\R^d),
\end{equation*}
is called the \emph{free Lie algebra} $\mathfrak{g}(\R^d)$ over $\R^d$, see e.g.~\cite[Section~7.3]{Friz2010}, which is a Lie subalgebra of $T_0((\R^d))$. 
For $N\in \N$, we define the free step-$N$ nilpotent Lie algebra $\mathbf {\mathfrak g}^N(\R^d)\subset T_0^N(\R^d)$ with
\begin{equation*}
  \mathbf {\mathfrak g}^N(\R^d):=\{0\}\oplus\R^d\oplus[\R^d,\R^d]\oplus\ldots\oplus\underbrace{[\R^d,[\ldots,[\R^d,\R^d]]]}_{(N-1)\text{ brackets}},
\end{equation*}
where $(\mathbf g,\mathbf h)\mapsto [\mathbf g,\mathbf h]:=\mathbf g\otimes\mathbf h-\mathbf h\otimes\mathbf g\in T_0^N(\R^d)$ denotes the Lie bracket for $\mathbf g, \mathbf h\in T^N(\R^d)$, see \cite[Chapter~7.3.2 and Definition~7.25]{Friz2010}.

The image $G^N(\R^d):=\exp_\otimes^N(\mathbf{\mathfrak g}^N(\R^d))$ is a (closed) sub-Lie group of $(T_1^N(\R^d),\otimes)$, called the free nilpotent group of step $N$ over $\R^d$, see \cite[Theorem~7.30]{Friz2010}, where the exponential map is defined by 
\begin{equation*}
  \exp_{\otimes}^N \colon T_0^N(\R^d) \to T_1^N(\R^d), \qquad \mathbf{a} \mapsto 1 + \sum_{k=1}^{N} \frac{1}{k!} \mathbf{a}^{\otimes k}.
\end{equation*}
 We further define the set of group-like elements
\begin{equation*} 
 G((\R^d)) := \{\mathbf a\in T((\R^d)): (a^{(0)},a^{(1)},\ldots,a^{(N)})\in G^N(\R^d)\text{ for all } N\},
\end{equation*} 
 which is a subgroup of $T_1((\R^d))$.
 
 In fact, $(G((\R^d)),\otimes)$ is a group with unit element $(1,0,\dots,0,\dots)$. 

We define $I:=(i_1,\ldots,i_n)$ as a $n$-dimensional multi-index of non-negative integers, i.e. $i_j\in\{1,\ldots,d\}$ for every $j\in\{1,2,\ldots,n\}$. Note that $|I|:=n$ and the empty index is given by $I:=\emptyset$ with $|I|=0$. For $n\ge 1$ or $n\ge2$, we write $I^\prime:=(i_1,\ldots,i_{n-1})$ and $I^{\prime\prime}:=(i_1,\ldots,i_{n-2})$, respectively. Moreover, for each $|I|\ge 1$, we set $e_I:=e_{i_1}\otimes\cdots\otimes e_{i_n}$. This allows us to write $\mathbf{a} \in T((\R^d))$ (and $\mathbf{a} \in T(\R^d)$) as
\begin{equation*}
  \mathbf{a} = \sum_{|I| \geq 0} \langle e_I,\mathbf a\rangle e_I,
\end{equation*}
where $\langle \cdot,\cdot\rangle$ is defined as the inner product of $(\R^d)^{\otimes n}$ for each $n\ge 0$. 

For two multi-indices $I = (i_1, \ldots, i_{|I|})$ and $J = (j_1, \ldots, j_{|J|})$ with entries in $\{1,\ldots,d\}$, the shuffle product is recursively defined by
\begin{equation*}
  e_I \shuffle e_J := (e_{I'} \shuffle e_J) \otimes e_{i_{|I|}} + (e_I \shuffle e_{J'}) \otimes e_{j_{|J|}},
\end{equation*}
with $e_I \shuffle e_\emptyset := e_\emptyset \shuffle e_I := e_I$. For all $\mathbf{a} \in G((\R^d))$, the shuffle product property holds, i.e., for two multi-indices $I = (i_1, \ldots, i_{|I|})$ and $J = (j_1, \ldots, j_{|J|})$, it holds that
\begin{equation*}
  \langle e_I, \mathbf{a} \rangle \langle e_J, \mathbf{a} \rangle = \langle e_I \shuffle e_J, \mathbf{a} \rangle.
\end{equation*}

\subsection{Essentials on rough path theory}

Let $(E,\|\,\cdot\,\|_E)$ be a normed space. For $\alpha\in (0,1]$, the $\alpha$-H{\"o}lder norm of a path $X\in C([0,T];E)$ is given by
\begin{equation*}
  \|X\|_\alpha:=\sup_{0\le s<t\le T}\frac{\|X_t-X_s\|_E}{|t-s|^\alpha}.
\end{equation*}
We write $C^\alpha([0,T];E)$ for the space of all paths $X\in C([0,T];E)$ which satisfy $\|X\|_\alpha<\infty$. The $1$-variation of a continuous path $X\colon [0,T]\to E$ is defined by
\begin{equation*}
  \|X\|_{1\textup{-var}}:=\sup_{\mathcal D\subset [0,T]}\sum_{t_i\in\mathcal D}\|X_{t_i}-X_{t_{i-1}}\|_{E},
\end{equation*}
where the supremum is taken over all partitions $\mathcal D=\{0=t_0<t_1<\cdots<t_n=T\}$ of the interval $[0,T]$ and $\sum_{t_i\in\mathcal D}$ denotes the summation over all points in $\mathcal D$. If $\|X\|_{1\textup{-var}}<\infty$, we say that $X$ is of bounded variation or of finite $1$-variation on $[0,T]$. The space of continuous paths of bounded variation on $[0,T]$ with values in $E$ is denoted by $C^{1\textup{-var}}([0,T];E)$. 

Let $\Delta_T:=\{(s,t)\in [0,T]^2:s\le t\}$ be the standard $2$-simplex. For $\alpha\in(0,1]$ and a two-parameter function $\X^{(2)}\colon \Delta_T \to E$, we define
\begin{equation*}
  \|\X^{(2)}\|_\alpha:=\sup_{0\le s < t\le T}\frac{\|\X_{s,t}^{(2)}\|_E}{|t-s|^\alpha},\quad (s,t)\in\Delta_T,
\end{equation*}
and denote by $C_2^{\alpha}(\Delta_T;E)$ the space of all continuous functions $\X^{(2)}\colon \Delta_T\to E$ which satisfy $\|\X^{(2)}\|_\alpha<\infty$. In what follows, for a path $X\in C([0,T];\R^d)$, we will often use the shorthand notation
\begin{equation*}
  X_{s,t}:=X_t-X_s, \quad (s,t)\in\Delta_T.
\end{equation*}

Let $\alpha\in(\frac{1}{3},\frac{1}{2}]$ and $X\in C^{\alpha}([0,T];\R^d)$. A path $Y\in C^{\alpha}([0,T];\R^m)$ is said to be controlled by $X$ if there exists a path $Y^{\prime}\in C^{\alpha}([0,T];\mathcal{L}(\R^d;\R^m))$ such that the remainder term $R^Y\in C^{2\alpha}_2([0,T];\R^m)$ given through the relation
\begin{equation*}
  Y_{s,t}=Y_s^\prime X_{s,t}+R_{s,t}^Y,\quad (s,t)\in\Delta_T,
\end{equation*}
satisfies $\|R^Y\|_{2\alpha}<\infty.$ The path $Y^\prime$ is called Gubinelli derivative of $Y$. The set of controlled paths $(Y,Y^\prime)$ is denoted by $\mathcal{D}_X^{2\alpha}([0,T];\R^m)$, see \cite[Definition~4.6]{Friz2020}.

\medskip

For a path $X\in C^{1\textup{-var}}([0,T];\R^d)$ of finite variation, we denote by $\X^N$ the signature truncated at level $N$, which is given by
\begin{equation*}
  \X_{s,t}^N:=\Bigl(1,\int_{s<u<t}\dd X_u,\ldots, \int_{s<u_1<\ldots<u_N<t}\dd X_{u_1}\otimes\cdots\otimes\dd X_{u_N}\Bigr)\in T^N(\R^d),
\end{equation*}
for $0\le s\le t\le T$, where the integrals are defined in a classical Riemann--Stieltjes sense. The signature $\X_{s,t}$ of the path $X$ on $[s,t]$, given by
\begin{equation*}
  \X_{s,t} :=(1,X_{s,t},\X_{s,t}^{(2)},\ldots,)\in T((\R^d)),
\end{equation*}
for $0\le s\le t\le T$, where 
\begin{equation*}
  \X_{s,t}^{(n)}:=\int_{s<u_1<\ldots<u_n<t}\dd X_{u_1}\otimes\cdots\otimes\dd X_{u_n}
\end{equation*}
denotes the $n$-th component of $\X_{s,t}$. For $s=0$ we simply write $\X_t$.
 
Furthermore, the Carnot--Carathéodory norm $\|\,\cdot\,\|_{cc}$ on $G^N(\R^d)$ is defined by
\begin{equation*}
  \|\mathbf g\|_{cc}:=\inf\biggl\{\int_{0}^{T}|\dd X_t|\,:\,X\in C^{1\textup{-var}}([0,T];\R^d) \text{ such that } \bbX_T^N=\mathbf g\biggr\},
\end{equation*}
for $\mathbf g\in G^N(\R^d)$, which induces a metric via
\begin{equation*}
  d_{cc}(\mathbf g, \mathbf h):=\|\mathbf{g}^{-1}\otimes \mathbf h\|_{cc},\quad \text{for }\mathbf g, \mathbf h \in G^N(\R^d).
\end{equation*}

For $\alpha\in(0,1]$, a continuous path $\bX\colon [0,T]\to G^{\lfloor 1/ \alpha\rfloor}(\R^d)$ of the form
\begin{equation*}
  [0,T]\ni t\mapsto \bX_t:=\Bigl(1,\X^{(1)}_t,\bbX_t^{(2)},\ldots,\bbX_t^{(\lfloor 1/ \alpha\rfloor)}\Bigr)\in G^{\lfloor 1/ \alpha\rfloor}(\R^d)
\end{equation*}
with $\bX_0:=\mathbf{1}:=(1,0,\ldots,0)\in G^{\lfloor 1/ \alpha\rfloor}(\R^d)$ is called weakly geometric $\alpha$-H{\"o}lder rough path if the $\alpha$-H{\"o}lder norm
\begin{equation*}
  \|\bX\|_{cc,\alpha}:=\sup_{\overset{s,t\in[0,T]}{s<t}}\frac{d_{cc}(\bX_s,\bX_t)}{|s-t|^\alpha}<\infty,
\end{equation*}
where $\lfloor 1/ \alpha\rfloor:= \max \{k\in \mathbb{Z}\,:\, k\leq1/\alpha\}$. We denote by $C^\alpha([0,T];G^{\lfloor 1/ \alpha\rfloor}(\R^d))$ the space of such weakly geometric $\alpha$-H{\"o}lder rough paths, which we equip with the metric
\begin{equation*}
  d_{cc,\alpha}(\bX,\bY):=\sup_{\overset{s,t\in[0,T]}{s<t}}\frac{d_{cc}(\bX_{s,t},\bY_{s,t})}{|s-t|^{\alpha}},
\end{equation*}
for $\bX,\bY\in C^\alpha([0,T];G^{\lfloor 1/ \alpha\rfloor}(\R^d))$, where $\bX_{s,t}:=\bX_s^{-1}\otimes \bX_t\in G^{\lfloor 1/ \alpha\rfloor}(\R^d)$. Moreover, we introduce the metric
\begin{equation*}
  d_{cc,\infty}(\bX,\bY):=\sup_{t\in[0,T]}d_{cc}(\bX_t,\bY_t),
\end{equation*}
for $\bX,\bY\in C^\alpha([0,T];G^{\lfloor 1/ \alpha\rfloor}(\R^d)).$

The space of geometric $\alpha$-H{\"o}lder rough paths, denoted by
\begin{equation*}
  C^{0,\alpha}([0,T];G^{\lfloor 1/\alpha\rfloor}(\R^d)),
\end{equation*}
is defined as the closure of canonical lifts of smooth paths with respect to the $\alpha$-H{\"o}lder norm $\|\,\cdot\,\|_{cc,\alpha}$, that is, for every $\bX\in C^{0,\alpha}([0,T];G^{\lfloor 1/\alpha\rfloor}(\R^d))$ there exist a sequence of smooth paths $X^n$ such that
\begin{equation*}
  d_{cc,\alpha}(\X^n,\bX)\to 0\text{ as } n\to\infty,
\end{equation*}
where $\X^n$ is the $\lfloor 1/\alpha\rfloor$-step signature of $X^n$. The space $ C^{0,\alpha}([0,T];G^{\lfloor 1/\alpha\rfloor}(\R^d))$ is equipped with the metric
\begin{equation*}
  d_{cc,\alpha^{\prime}}(\bX,\bY):=\sup_{\overset{s,t\in[0,T]}{s<t}}\frac{d_{cc}(\bX_{s,t},\bY_{s,t})}{|s-t|^{\alpha^{\prime}}},
\end{equation*}
for $\bX,\bY\in C^{0,\alpha}([0,T];G^{\lfloor 1/ \alpha\rfloor}(\R^d))$ and $0\le\alpha^{\prime}\le \alpha$, where $\bX_{s,t}:=\bX_s^{-1}\otimes \bX_t\in G^{\lfloor 1/ \alpha\rfloor}(\R^d)$.

The space of geometric $\alpha$-H{\"o}lder rough paths $C^{0,\alpha}([0,T];G^{\lfloor 1/ \alpha\rfloor}(\R^d))$ is a closed subset of the space of weakly geometric $\alpha$-H{\"o}lder rough paths $C^{\alpha}([0,T];G^{\lfloor 1/ \alpha\rfloor}(\R^d))$ and thus complete, see \cite[Definition~8.19]{Friz2010}. The distinction between geometric and weakly geometric rough paths is discussed in detail in \cite{Friz2006}.

\medskip

Let us introduce the truncated signature at level $N>\lfloor 1/\alpha\rfloor$ of a (weakly) geometric $\alpha$-H{\"o}lder rough path $\bX\in C^{0,\alpha}([0,T];G^{\lfloor 1/\alpha\rfloor}(\R^d))$ as the unique Lyons' extension, see e.g. \cite[Theorem~9.5, Corollary~9.11~(ii)]{Friz2010}, yielding a path $\X^N\colon [0,T]\to G^N(\R^d)$. Then, $\X^N$ has finite $\alpha$-H{\"o}lder norm $\|\,\cdot\,\|_{cc,\alpha}$ and starts with the unit element $\mathbf 1:=(1,0,\ldots,0)\in G^N(\R^d)$, and the signature of $\bX$ is given by
\begin{equation*}
  [0,T]\ni t\mapsto \X_t=\Bigl(1,\X^{(1)}_t,\X^{(2)}_t,\ldots,\X_t^{(\lfloor 1/\alpha\rfloor)},\ldots,\X_t^{(N)},\ldots\Bigl).
\end{equation*}

\begin{remark}\label{rem: weak topology}
  Note that we equip the space of geometric $\alpha$-H{\"o}lder rough paths with a weaker topology than the norm topology, to obtain an admissible weight function, i.e., the closed unit ball is then compact (the pre-image $K_R=\psi^{-1}((0,R])$ is then compact w.r.t. the weaker topology). More precisely, in \cite[p.~37]{Cuchiero2024} it is discussed that the space $C^{0,\alpha}([0,T];G^{\lfloor 1/\alpha\rfloor}(\R^d))$ equipped with the metric $d_{cc,\alpha^\prime}$ and the weight function
  \begin{equation*}
    \psi(\bX):=\exp(\beta\|\bX\|_{cc,\alpha}^\gamma)
  \end{equation*}
  is a weighted space for some $\beta>0$ and $\gamma\ge\lfloor 1/\alpha\rfloor$, which follows from the compact embedding
  \begin{equation*}
    (C^{0,\alpha}([0,T];G^{\lfloor 1/\alpha\rfloor}(\R^d)),d_{cc,\alpha})\hookrightarrow (C^{0,\alpha^\prime}([0,T];G^{\lfloor 1/\alpha\rfloor}(\R^d)),d_{cc,\alpha^\prime})
  \end{equation*}
  for $0< \alpha^\prime<\alpha\le 1$, see \cite[Remark~A.7~(i) and p.~37]{Cuchiero2024}. We refer to \cite{Cuchiero2024} for an extensive discussion of the weaker topologies on the space of geometric $\alpha$-H{\"o}lder rough paths, including the weak-$\ast$-topology.
\end{remark}

\section{Global approximation with rough path signatures}\label{sec: global approx}

In this section, we establish $L^p$-universal approximation theorems for linear functionals acting on signatures of time-extended rough paths. Our approach builds on the universal approximation theorem for weighted spaces proven in \cite{Cuchiero2024}. We begin by deriving a universal approximation result for $p$-integrable functionals on the rough path space and then present an analogous theorem for $p$-integrable non-anticipative functionals.

\subsection{General functionals}

In this subsection, we consider the space $(\widehat{C}_{d,T}^{\alpha},\cB(\widehat{C}_{d,T}^{\alpha}))$ of time-extended rough paths, which is defined as
\begin{equation*}
  \widehat{C}_{d,T}^{\alpha}:=\Bigl\{\hbX\in C^{0,\alpha}([0,T];G^{\lfloor 1/ \alpha\rfloor}(\R^{d+1})): \langle e_0,\hbX_t\rangle:=t \text{ for all }t\in[0,T]\Bigr\},
\end{equation*}
that is, the subspace of $C^{0,\alpha}([0,T]; G^{\lfloor 1/\alpha \rfloor}(\mathbb{R}^{d+1}))$, where the $0$-th coordinate represents the running time, for $\alpha \in (0,1)$. The space $(\widehat{C}_{d,T}^{\alpha},\cB(\widehat{C}_{d,T}^{\alpha}))$ is equipped with the $\alpha^\prime$-H{\"o}lder metric $d_{cc,\alpha^\prime}$ for some $0<\alpha^\prime<\alpha$ and let $\nu$ be a finite Borel measure on $(\widehat{C}_{d,T}^{\alpha},\cB(\widehat{C}_{d,T}^{\alpha}))$, i.e. $\nu(\widehat{C}_{d,T}^{\alpha})<\infty$, where $\cB(\widehat{C}_{d,T}^{\alpha})$ denotes the Borel $\sigma$-algebra on $\widehat{C}_{d,T}^{\alpha}$. Moreover, in what follows, we work with the weight function
\begin{equation}\label{eq: weight function}
  \psi(\hbX):=\exp(\beta\|\hbX\|_{cc,\alpha}^\gamma)
\end{equation}
for some $\beta>0$ and $\gamma\ge\lfloor 1/\alpha\rfloor$. Note that, by Remark~\ref{rem: weak topology}, the space $\widehat{C}_{d,T}^\alpha$ equipped with $d_{cc,\alpha^\prime}$ is a weighted space.

\begin{remark}
  The signature of a (rough) path determines the path only up to so-called tree-like equivalence; see \cite{Hambly2010,Boedihardjo2016}. By augmenting the path with time in the $0$-th coordinate, the signature of the resulting time-extended (rough) path uniquely determines the original path up to translation. This property is essential for applying a Stone--Weierstrass theorem in order to obtain universal approximation results for linear functionals on signatures. Although adding time is a natural and commonly used choice, this uniqueness feature can be achieved by extending a (rough) path with any strictly monotone one-dimensional path.
\end{remark}

\begin{remark}
  We emphasize that, in contrast to \cite{Cuchiero2024}, we do not work with the space of weakly geometric $\alpha$-H{\"o}lder rough paths, but rather with the space of geometric $\alpha$-H{\"o}lder rough paths. The reason is that the latter, when equipped with its natural $\alpha$-H{\"o}lder rough path topology, forms a Polish space. Consequently, a geometric $\alpha$-H{\"o}lder rough path $\bX$ can be regarded as a $C^{0,\alpha}([0,T]; G^{\lfloor 1/\alpha \rfloor}(\R^d))$-valued random variable, and its law $\mu_{\bX}$ is then a Borel measure on the corresponding Borel $\sigma$-algebra; see \cite[Appendix~A1]{Friz2010}.  In the weighted approximation results below, we shall use the weaker topology induced by $d_{cc,\alpha^\prime}$, where $0<\alpha^\prime<\alpha$. Since the embedding from the $\alpha$-H{\"o}lder space to the $\alpha^\prime$-H{\"o}lder space is continuous, every $d_{cc,\alpha^\prime}$-Borel set is also Borel with respect to the natural $\alpha$-H{\"o}lder topology. Thus the law of a geometric $\alpha$-H{\"o}lder rough path, initially defined on the natural Polish Borel space, also defines a Borel probability measure on the measurable space generated by the weaker $\alpha^\prime$-H{\"o}lder topology.
\end{remark}

To derive $L^p$-universal approximation theorems for linear functionals acting on signatures of time-extended rough paths, we rely on a slight modification of the universal approximation result for weighted spaces established in \cite[Theorem~5.4]{Cuchiero2024}.

\begin{proposition}[Universal approximation theorem on $\mathcal B_\psi(\widehat{C}_{d,T}^\alpha)$]\label{prop: weighted UAT 1}
  Let $\psi$ be the weight function given in \eqref{eq: weight function}. Then, the linear span of the set
  \begin{equation*}
    \Bigl\{\hbX\mapsto\langle e_I,\hbbX_T\rangle: I\in\{0,\ldots,d\}^N,N\in\N_0\Bigr\}
  \end{equation*}
  is dense in $\cB_\psi(\widehat{C}_{d,T}^\alpha)$, i.e., for every map $f\in\cB_\psi(\widehat{C}_{d,T}^\alpha)$ and every $\epsilon>0$ there exists a linear function $\boldsymbol\ell \colon T((\R^{d+1}))\to \R$  of the form $\hbbX_T\mapsto\boldsymbol\ell(\hbbX_T):=\sum_{|I|\le N}\ell_I\langle e_I,\hbbX_T\rangle$, for some $N\in\N_0$ and $\ell_I\in\R$, such that
  \begin{equation*}
    \sup_{\hbX\in \widehat{C}_{d,T}^\alpha}\frac{|f(\hbX)-\boldsymbol\ell(\hbbX_T)|}{\psi(\hbX)}<\epsilon.
  \end{equation*}
\end{proposition}

\begin{proof}
  The proof follows that of \cite[Theorem~5.4]{Cuchiero2024} verbatim, with the sole modification of replacing the weakly geometric rough path space by the geometric rough path space. The weighted-space hypothesis, in particular the admissibility of the weight, is ensured by Remark~2.1. The remaining ingredients used in the Stone--Weierstrass argument are unchanged under this replacement: signature coordinates are continuous, form an algebra by the shuffle identity, and the time-extended signature coordinates separate points also on the geometric rough path space. Hence, the assumptions of the weighted real-valued Stone--Weierstrass theorem \cite[Theorem~3.9]{Cuchiero2024} are verified exactly as in \cite[Theorem~5.4]{Cuchiero2024}. We therefore omit the details to avoid unnecessary repetition.
\end{proof}

We are now in a position to state a global universal approximation theorem for linear functionals acting on signatures of time-extended rough paths in the space $L^p(\widehat{C}_{d,T}^\alpha)$.

\begin{theorem}[$L^p$-universal approximation theorem on $\widehat{C}_{d,T}^\alpha$]\label{thm:Lpmain}
  Let $\psi$ be the weight function given in \eqref{eq: weight function},  $p\ge 1$, and $\int_{\widehat{C}_{d,T}^\alpha}\psi^p\dd\nu<\infty$. Moreover, we consider the set
  \begin{equation*}
    \mathcal L:=\Bigl\{f_\ell\,:\,f_\ell\colon\hbX\mapsto\boldsymbol\ell(\hbbX_T)=\sum_{|I|\le N}\ell_I\langle e_I,\hbbX_T\rangle,\,\ell_I\in\R,\,N\in\N_0,\,\hbX\in\widehat{C}_{d,T}^\alpha\Bigr\}.
  \end{equation*}
  Then, for every $f\in L^p(\widehat{C}_{d,T}^\alpha)$ and for every $\epsilon>0$, there exists a functional $f_\ell\in\mathcal L$ such that
  \begin{equation*}
    \|f-f_\ell\|_{L^p(\widehat{C}_{d,T}^\alpha)}<\epsilon.
  \end{equation*}
\end{theorem}

\begin{proof}
  Let $f\in L^p(\widehat{C}_{d,T}^\alpha,\nu)$ and fix $\epsilon>0$.

  \textit{Step~1.} For any $K>0$, we can define the function $f_K(x):=1_{\{|f(x)|\le K\}}(x)f(x)$ for which we have $\|f-f_K\|_{L^p(\widehat{C}_{d,T}^\alpha)}\to 0$ as $K\to \infty$ by dominated convergence. Therefore, there is a $K^{\epsilon}>0$ such that
  \begin{equation*}
    \|f-f_{K^{\epsilon}}\|_{L^p(\widehat{C}_{d,T}^\alpha)}\le\frac{\epsilon}{3}.
  \end{equation*}

  \textit{Step~2.} By Lusin's theorem \cite[Theorem~2.5.17]{Denkowski2003}, there is a closed set $C^\epsilon\subset \widehat{C}_{d,T}^\alpha$, such that $f_{K^\epsilon}$ restricted to $C^\epsilon$ is continuous and $\nu(\widehat{C}_{d,T}^\alpha\setminus C^\epsilon)\le \frac{\epsilon^p}{(6 K^\epsilon)^p}$. By Tietze's extension theorem \cite[Theorem~3.6.3]{friedman1982}, there is a continuous extension $f^\epsilon\in C_b(\widehat{C}_{d,T}^\alpha;[-K^\epsilon,K^\epsilon])$ of $f_{K^\epsilon},$ such that
  \begin{equation*}
    \|f_{K^\epsilon}-f^\epsilon\|_{L^p(\widehat{C}_{d,T}^\alpha)}^p=\int_{\widehat{C}_{d,T}^\alpha\setminus C^\epsilon}|f_{K^\epsilon}-f^\epsilon|^p\dd\nu\le(2K^\epsilon)^p\nu(\widehat{C}_{d,T}^\alpha\setminus C^\epsilon)\le \Bigl(\frac{\epsilon}{3}\Bigr)^p.
  \end{equation*}

  \textit{Step~3.} Moreover, since by the definition of the weighted function space $\cB_\psi$ it holds that $C_b(\widehat{C}_{d,T}^\alpha)\subseteq \cB_\psi(\widehat{C}_{d,T}^\alpha)$, by Proposition~\ref{prop: weighted UAT 1} we can approximate $f^\epsilon$ by a linear function on the signature. More precisely, set $M:=\int_{\widehat{C}^{\alpha}_{d,T}}\psi^p\dd\nu<\infty$, then we have
  \begin{equation*}
    \|{f}^\epsilon-f_\ell\|^p_{\cB_\psi(\widehat{C}_{d,T}^\alpha)}=\Bigl(\sup_{\hbX\in\widehat{C}_{d,T}^\alpha}\frac{|{f}^\epsilon(\hbX)-\boldsymbol{\ell}(\hbbX_T)|}{\psi(\hbX)}
    \Bigr)^p< \frac{\epsilon^p}{3^p M}.
  \end{equation*}
  Hence, we get
  \begin{equation*}
    \|{f}^\epsilon-f_\ell\|^p_{L^p(\widehat{C}_{d,T}^\alpha)}\le\int_{\widehat{C}_{d,T}^\alpha}\psi^p\dd\nu~\|{f}^\epsilon-f_\ell\|^p_{\cB_\psi(\widehat{C}_{d,T}^\alpha)}< \Bigl(\frac{\epsilon}{3} \Bigr)^p.
  \end{equation*}

  Hence, combining Step 1-3 reveals that
  \begin{equation*}
    \|f-f_\ell\|_{L^p(\widehat{C}_{d,T}^\alpha)}\le\|f-f_{K^\epsilon}\|_{L^p(\widehat{C}_{d,T}^\alpha)}+\|f_{K^\epsilon}-f^\epsilon\|_{L^p(\widehat{C}_{d,T}^\alpha)}+
    \|{f}^\epsilon-f_\ell\|_{L^p(\widehat{C}_{d,T}^\alpha)}< \epsilon,
  \end{equation*}
  which concludes the proof.
\end{proof}

\begin{remark}
  Note that the integrability condition $\int_{\widehat{C}_{d,T}^\alpha} \psi^p \,\mathrm{d}\nu < \infty$, with the weight function $\psi(\hbX) = \exp\!\bigl(\beta \|\hbX\|_{cc,\alpha}^{\gamma}\bigr)$, corresponds to an exponential moment condition. Indeed, it is equivalent to
  \begin{equation*}
    \int_{\widehat C^\alpha_{d,T}}
    \exp\bigl(
        p\beta
        \|\hbX\|_{cc,\alpha}^{\gamma}
    \bigr)
    \dd\nu(\hbX)
    <\infty.
  \end{equation*}
  Thus, if $\nu$ is the law of a random rough path $\hbX$, this condition reads
  \begin{equation*}
    \E\Bigl[
        \exp\bigl(
            p\beta
            \|\hbX\|_{cc,\alpha}^{\gamma}
        \bigr)
    \Bigr]<\infty.
  \end{equation*}
  In particular, for Gaussian rough paths, including (fractional) Brownian motions, considered in Section~\ref{sec: gaussian process}, this condition follows from the Gaussian tail estimates for the rough path norm when choosing $\beta>0$ sufficiently small.
\end{remark}

\subsection{Non-anticipative functionals}

In this subsection, we derive a global universal approximation theorem on the space of stopped $\alpha$-H{\"o}lder rough paths. To that end, for $\alpha \in (0,1)$ we consider
\begin{equation*}
  \widehat{C}_{d,t}^{\alpha}:=\Bigl\{\hbX_{[0,t]}\in C^{0,\alpha}([0,T];G^{\lfloor 1/\alpha\rfloor}(\R^{d+1})):  \langle e_0,\hbX_s\rangle:=s \text{ for all } s\in[0,t] \Bigr\},
\end{equation*}
where $\hbX_{[0,t]}$ stands for the rough path $\hbX$, which is defined on $[0,T]$, restricted to the sub-interval $[0,t]$, for $t\in [0,T]$. Furthermore, we require the notion of stopped rough paths. For related definitions, we refer, for example, to \cite{Kalsi2020,Bayer2025} in the rough path setting and to \cite{Cuchiero2024b} in a rough semimartingale framework. We also note that spaces of stopped paths already appear in the context of functional It{\^o} calculus; see \cite{Cont2013,Dupire2019}.

\begin{definition}\label{def: stopped rough path}
  Let $\alpha\in(0,1]$, $t\in[0,T]$, and let $\hbX_{[0,t]}\in\widehat{C}_{d,t}^{\alpha}$ be a geometric $\alpha$-H{\"o}lder rough path. We define the stopped rough path at time $t$, $\hbX^t_{[0,T]}\in\widehat{C}_{d,T}^{\alpha}$, as follows.

  Set $N := \lfloor 1/\alpha\rfloor$. By geometricity, there exists a sequence of smooth time-extended paths $\widehat X^n_s := (s,X^n_s)$ on $[0,t]$ such that their canonical lifts $\hbbX^{n}$ (i.e.~their signatures truncated at level $N$) converge to $\hbX$ on $[0,t]$ in the $\alpha$-H{\"o}lder rough path metric $d_{cc,\alpha}$. For $r\in[0,T]$ we define the stopped paths
  \begin{equation*}
    \widehat X^{n,t}_r := (r,X^{n,t}_r):=(r,X^n_{r\wedge t}), \qquad r\in[0,T],
  \end{equation*}
  i.e.~the time-extension is not stopped, and let $\hbbX^{n,t}$ be their canonical lifts on $[0,T]$. We then set
  \begin{equation*}
    \hbX^t_{[0,T]} := \lim_{n\to\infty} \hbbX^{n,t}_{[0,T]},
  \end{equation*}
  where the limit is taken in $d_{cc,\alpha}$. In particular, $(\hbX^t)_s = \hbX_s$ for all $s\in[0,t]$.
\end{definition}

This construction is standard in the related rough path literature. For instance, it is used in \cite{Bayer2024}, where the well-definedness of the stopped rough path extension is stated. To make the present paper self-contained, we provide the details in Appendix~\ref{sec: appendix}. In particular, we verify the existence of the above limit and its independence of the chosen smooth approximating sequence.

\begin{definition}
  The space $\Lambda_T^\alpha$ of stopped geometric $\alpha$-H{\"o}lder rough paths is defined by
  \begin{equation*}
    \Lambda_T^\alpha:=\bigcup_{t\in[0,T]}\widehat C_{d,t}^\alpha
  \end{equation*}
  and equipped with the metric
  \begin{equation*}
    d_{\Lambda,\alpha^\prime}(\hbX_{[0,t]},\hbY_{[0,s]})=|t-s|+d_{cc,\alpha^\prime}(\hbX^t_{[0,T]},\hbY^s_{[0,T]}),
  \end{equation*}
  for some $ 0<\alpha^\prime<\alpha$.
\end{definition}

\begin{remark}\label{rem: quotient topology}
  The form of the metric $d_{\Lambda,\alpha^\prime}$ is inspired by \cite{Cont2013} and by the  stopped rough path metric introduced in \cite[Definition~3.1]{Bayer2024}. The latter is stated, for $s\le t$, by comparing a rough path on $[0,t]$ with the stopped extension of a rough path on $[0,s]$ over the interval $[0,t]$.
 
  In the present paper, we use a symmetric fixed-horizon version which is defined for arbitrary pairs $\hbX_{[0,t]},\hbY_{[0,s]}\in\Lambda_T^\alpha$. This convention is convenient for our purposes, since all stopped rough paths are represented by their stopped extensions on the same time interval. It also makes the metric property immediate: the term $|t-s|$ is the usual metric on $[0,T]$, and $d_{cc,\alpha^\prime}$ is a metric on the fixed-horizon rough path space; hence their sum is a metric on $\Lambda_T^\alpha$. The resulting topology is the natural analogue of the topology considered in \cite[Lemma~A.1]{Bayer2024}, with the $p$-variation distance replaced by the $\alpha^\prime$-H{\"o}lder rough path distance; see also \cite[Remark~2.2]{Bayer2025} for the H{\"o}lder setting. Since our fixed-horizon formulation differs slightly from the one used there, we include a self-contained proof in Appendix~\ref{sec: appendix}. In particular, we show that the topology induced by $d_{\Lambda,\alpha^\prime}$ coincides with the final topology induced by
  \begin{equation*}
    \phi\colon [0,T]\times \widehat C^\alpha_{d,T}\to\Lambda_T^\alpha, \qquad \phi(t,\hbX)=\hbX_{[0,t]}.
  \end{equation*}
  Moreover, we conclude that $\Lambda_T^\alpha$ is Polish when equipped with the $d_{\Lambda,\alpha}$ metric.
\end{remark}

To obtain a global universal approximation result on $\Lambda_T^\alpha$, we shall first verify that $(\Lambda_T^\alpha, \psi)$ forms a weighted space. For this purpose, we consider the weight function
\begin{equation}\label{eq:weight function non-anticipative}
  \psi(\hbX_{[0,t]}):=\exp(\beta\|\hbX^t_{[0,T]}\|_{cc,\alpha}^\gamma), \quad\hbX_{[0,t]}\in \Lambda_T^{\alpha},
\end{equation}
for some $\beta>0$ and $\gamma\ge\lfloor 1/\alpha\rfloor$.

\begin{lemma}\label{lem: Lambda is weighted space}
  Let $0<\alpha^{\prime}<\alpha<1$ and suppose that $\psi$ is defined as in \eqref{eq:weight function non-anticipative}. Then, $K_R:=\psi^{-1}((0,R])=\{\hbX_{[0,t]}\in\Lambda_T^\alpha: \psi(\hbX_{[0,t]})\le R\}$ is compact with respect to the final topology and $(\Lambda_T^\alpha,\psi)$ is a weighted space.
\end{lemma}

\begin{proof}
  First observe that by the definition of the quotient map  $\phi$, we have
  \begin{equation*}
    K_R=\phi\Bigl([0,T]\times \{\hbX^{t}_{[0,T]}\in\widehat{C}_{d,T}^\alpha: \psi(\hbX_{[0,t]})\le R\}\Bigr).
  \end{equation*}
  Since $\phi$ is continuous, we only need to show that
  \begin{equation*}
    [0,T]\times \{\hbX^{t}_{[0,T]}\in \widehat{C}_{d,T}^\alpha: \psi(\hbX_{[0,t]})\le R\}
  \end{equation*}
  is compact in $[0,T]\times\widehat{C}_{d,T}^\alpha$ to obtain the compactness of $K_R$.
          
  Therefore, observe that the sets $\{\hbX^{t}_{[0,T]}\in \widehat{C}_{d,T}^\alpha: \psi(\hbX_{[0,t]})\le R\}$ are equicontinuous and pointwise bounded. Using that geometric $\alpha$-H{\"o}lder rough path spaces are compactly embedded in geometric $\alpha^\prime$-H{\"o}lder rough path spaces for $\alpha^\prime <\alpha$ (cf. \cite{Cuchiero2024}), we obtain that the sets $\{\hbX^{t}_{[0,T]}\in \widehat{C}_{d,T}^\alpha: \psi(\hbX_{[0,t]})\le R\}$ are, by the Arzèla--Ascoli theorem, see e.g. \cite[Theorem~4.43]{Folland1999}, compact with respect to the $\alpha^\prime$-H{\"o}lder norm. Since by Lemma~\ref{lem: phi cts} $\phi$ is continuous, $K_R$ is also compact for any $R>0$ due to Tychonoff's theorem. Thus, $(\Lambda_T^\alpha,\psi)$ is a weighted space. See also \cite[Lemma~2.10]{Bayer2025} for a similar proof.
\end{proof}

\begin{definition}
  A map $f\colon \Lambda_T^\alpha\to\R$ is called a non-anticipative functional if $f$ is measurable. A map $f\colon \Lambda_T^\alpha\to\R$ is called continuous if $f$ is continuous with respect to the metric $d_{\Lambda,\alpha^{\prime}}$.
\end{definition}

With these preparations in place, we can establish a global universal approximation result on $\mathcal{B}_\psi(\Lambda_T^\alpha)$.

\begin{proposition}[Universal approximation theorem on $\cB_\psi(\Lambda_T^\alpha)$]\label{prop: weighted UAT 2}
  Let $\psi$ be defined as in \eqref{eq:weight function non-anticipative}. Then, the linear span of the set
  \begin{equation*}
    \Bigl\{\hbX_{[0,t]}\mapsto\langle e_I,\hbbX_{t}\rangle: I\in\{0,\ldots,d\}^N,N\in\N_0\Bigr\}
  \end{equation*}
  is dense in $\cB_\psi(\Lambda_T^\alpha)$, i.e., for every map $f\in\cB_\psi(\Lambda_T^\alpha)$ and every $\epsilon>0$ there exists a linear function $\boldsymbol\ell\colon T((\R^{d+1}))\to \R$ of the form $\hbbX_t\mapsto\boldsymbol\ell(\hbbX_t):=\sum_{|I|\le N}\ell_I\langle e_I,\hbbX_t\rangle$, for some $N\in\N_0$ and $\ell_I\in\R$, such that
  \begin{equation*}
    \sup_{\hbX_{[0,t]}\in \Lambda_T^\alpha}\frac{|f(\hbX_{[0,t]})-\boldsymbol\ell(\hbbX_t)|}{\psi(\hbX_{[0,t]})}<\epsilon.
  \end{equation*}
\end{proposition}

\begin{proof}
  First note that, since $(\Lambda_T^\alpha,\psi)$ is a weighted space by Lemma~\ref{lem: Lambda is weighted space}, we are able to apply the weighted real-valued Stone--Weierstrass theorem, stated in \cite[Theorem~3.9]{Cuchiero2024}. The proof follows the same general strategy as the proof of \cite[Theorem~5.4]{Cuchiero2024}. Compared to the ordinary path spaces considered there, the main modification is that we work on the stopped rough path space $\Lambda_T^\alpha$, so the evaluation time is variable and continuity has to be checked with respect to the stopped rough path metric. Moreover, the auxiliary space $\widetilde{\cA}$ is adapted to the stopped  setting: since the time coordinate equals $t$ and may vanish at $t=0$, we include the constant coordinate to obtain a nowhere-vanishing element.
  
We apply the weighted real-valued Stone--Weierstrass theorem to
  \begin{equation*}
    \cA := \spn\Bigl\lbrace \hbX_{[0,t]} \mapsto \langle e_I, \hbbX_{t} \rangle : I \in \lbrace 0,\ldots,d \rbrace^N, \, N \in \N_0 \Bigr\rbrace.
  \end{equation*}
  To verify the assumptions of the weighted Stone--Weierstrass theorem, we first show that $\cA\subseteq \cB_\psi(\Lambda_T^\alpha)$ is a subalgebra. Moreover, by the definition of point separating and nowhere vanishing algebras of $\psi$-moderate growth, see  \cite[Definition~3.4]{Cuchiero2024}, it is sufficient to exhibit a vector subspace $\widetilde{\cA}\subseteq\cA$ which is point separating, nowhere vanishing, and of $\psi$-moderate growth. We choose
  \begin{align}\label{EqThmUATHoelderProof1}
    \widetilde{\cA}
    &:= \spn\Bigl(
      \Bigl\{
        \hbX_{[0,t]}
        \mapsto
        \langle e_\emptyset,\hbbX_t\rangle
      \Bigr\}\nonumber\\
      &\qquad \cup
      \Bigl\{
        \hbX_{[0,t]} \mapsto
        \langle ( e_I \shuffle e_0^{\otimes k} ) \otimes e_0,
        \hbbX_t \rangle :
        \begin{matrix}
          k \in \N_0, \,
          N \in \{0,\ldots,\lfloor 1/\alpha \rfloor \},\\
          I \in \{0,\ldots,d\}^N
        \end{matrix}
      \Bigr\}
    \Bigr)\\
    &\subseteq \cA\nonumber.
  \end{align}

  In order to prove that $\cA \subseteq \mathcal{B}_\psi(\Lambda_T^\alpha)$ is a vector subspace, we fix some $a \in \cA$ of the form $\Lambda_T^\alpha \ni \hbX_{[0,t]} \mapsto a(\hbX_{[0,t]}) := \langle e_I, \hbbX_t \rangle \in \R$, for some $I \in \lbrace 0,\ldots,d \rbrace^N$ and $N \in \N_0$.

  We define
  \begin{equation*}
   d_{\Lambda,\infty}(\hbX_{[0,t]},\hbY_{[0,s]}):= |t-s|+d_{cc,\infty} (\hbX^t_{[0,T]}, \hbY^s_{[0,T]}).
  \end{equation*}
  Note that by \cite[Lemma~5.2]{Cuchiero2024}, on each sublevel set $K_R:=\psi^{-1}((0,R])$, the metrics $d_{\Lambda,\infty}$ and $d_{\Lambda,\alpha'}$ induce the same topology. Hence, in order to verify continuity on $K_R$, it is enough to work with $d_{\Lambda,\infty}$.
  Therefore, recall that by Remark~\ref{rem: quotient topology} and Lemma~\ref{lem: topologies} the topology on $(\Lambda_T^\alpha,d_{\Lambda,\infty})$ coincides with the final topology induced by the map
  \begin{equation*}
    \phi\colon [0,T]\times \widehat{C}_{d,T}^\alpha\to\Lambda_T^\alpha,\qquad \phi(t,\hbX)=\hbX_{[0,t]},
  \end{equation*}
  where here we equip $\widehat{C}_{d,T}^\alpha$ with the metric $d_{cc,\infty}$. Then, a map $f\colon \Lambda_T^\alpha\to\R$ is continuous if and only if the composition $f\circ \phi\colon [0,T]\times \widehat{C}_{d,T}^{\alpha}\to\R$ is continuous. Thus, it suffices to prove continuity of $\bar{a}:=a\circ \phi$. Therefore, we fix some $R > 0$ and observe that the pre-image $K_R := \psi^{-1}((0,R])$ is bounded with respect to $d_{\Lambda,\alpha}$.

  For $(t,\hbX)\in [0,T]\times \widehat{C}_{d,T}^\alpha$, we have
  \begin{equation*}
    \bar{a}(t,\hbX)=a(\phi(t,\hbX))=a(\hbX_{[0,t]})=\langle e_I,\hbbX_t\rangle.
  \end{equation*}
  Now, let $\widetilde{K}_{R}\subset \widehat{C}_{d,T}^\alpha$ be a subset bounded with respect to the $\alpha$-H{\"o}lder norm $\|\,\cdot\,\|_{cc,\alpha}$. Then, it follows from \cite[Corollary~10.40]{Friz2010} that the map
  \begin{equation*}
    (\widetilde{K}_R,d_{cc,\infty}) \ni \hbX \quad \mapsto \quad \hbbX^N \in (C^{0,\alpha}([0,T];G^{N}(\R^{d+1})),d_{cc,\infty})
  \end{equation*}
  is continuous on $\widetilde{K}_R$ with respect to $d_{cc,\infty}$. This together with the continuity of the evaluation map
  \begin{equation*}
    (C^{0,\alpha}([0,T];G^N(\R^{d+1})),d_{cc,\infty}) \ni \hbbX^N \quad \mapsto \quad \hbbX^N_t \in (G^N(\R^{d+1}),d_{cc})
  \end{equation*}
  shows that the map
  \begin{equation*}
    (\widetilde{K}_R,d_{cc,\infty}) \ni \hbX \quad \mapsto \quad \hbbX^N_t \in (G^N(\mathbb{R}^{d+1}),d_{cc})
  \end{equation*}
  is continuous on $\widetilde{K}_R$ with respect to $d_{cc,\infty}$. Then, it also follows that
  \begin{equation*}
    ([0,T]\times \widetilde{K}_R,d_{\textup{prod}})\ni (t,\hbX)\quad\mapsto\quad \hbbX_t^N\in (G^N(\R^{d+1}),d_{cc}),
  \end{equation*}
  is continuous on $[0,T]\times \widetilde K_R$ with respect to the product metric $d_{\textup{prod}}:=|\,\cdot - \cdot\,|+d_{cc,\infty}$. Further, since linear functions on the finite dimensional space $G^N(\R^{d+1})$ are continuous, it follows that the map
  \begin{equation*}
    ([0,T]\times \widetilde K_R,d_{\textup{prod}}) \ni (t,\hbX) \quad \mapsto \quad \bar{a}(t,\hbX)= \langle e_I, \hbbX_t \rangle \in \R
  \end{equation*}
  is continuous on $[0,T]\times \widetilde K_R$ with respect to the product metric $d_{\textup{prod}}$. We now choose
  \begin{equation*}
    \widetilde{K}_R=\{\hbX^t_{[0,T]}\in\widehat C_{d,T}^\alpha:\psi(\hbX_{[0,t]})\le R\},
  \end{equation*}
  which, is bounded with respect to $\|\,\cdot\,\|_{cc,\alpha}$. Then, by construction
  \begin{equation*}
    K_R=\phi\Bigl([0,T]\times \widetilde{K}_R \Bigr),
  \end{equation*}
  and the topology on $K_R$ is the final topology induced by $\phi_R := \phi|_{[0,T]\times \widetilde K_R}$. Since $\bar a|_{[0,T]\times\widetilde K_R} = a|_{K_R} \circ \phi_R$, is continuous, we then obtain that the map
  \begin{equation}\label{eq: cts}
    (K_R,d_{\Lambda,\infty}) \ni \hbX_{[0,t]} \quad \mapsto \quad  a(\hbX_{[0,t]}) = \langle e_I, \hbbX_t \rangle \in \mathbb{R}
  \end{equation}
  is continuous on $K_R$ with respect to $d_{\Lambda,\infty}$. Since $R > 0$ was chosen arbitrarily, this shows that $a\vert_{K_R} \in C(K_R)$, for all $R > 0$.

  Moreover, using the ball-box-estimate (see \cite[Proposition~7.49]{Friz2010}), we have
  \begin{equation*}
    \Vert g-h \Vert_{T^N(\R^{d+1})} \leq C_1 \max\left( d_{cc}(g,h) \max\Bigl(1, \Vert g \Vert_{cc}^{N-1} \right), d_{cc}(g,h)^N\Bigr)
  \end{equation*}
  for each $g,h \in G^N(\R^{d+1})$ and some constant $C_1 \geq 1$ and by choosing $g = \hbbX^N_0$ and $h = \hbbX^N_t$ we obtain for every $\hbX_{[0,t]} \in \Lambda_T^\alpha$ that
  \begin{equation*}
    \vert a(\hbX_{[0,t]}) \vert = \vert \langle e_I, \hbbX_t \rangle \vert \leq \Vert \hbbX^N_t \Vert_{T^N(\R^{d+1})} \leq \Vert \hbbX^N_t - \hbbX^N_0 \Vert_{T^N(\R^{d+1})} + 1 \leq C_1\Bigl( d_{cc}(\hbbX^N_t, \hbbX^N_0)^N + 2 \Bigr).
  \end{equation*}
  Using the inequality $d_{cc}(\hbbX^N_u,\hbbX^N_s) \leq C_{N,\alpha} d_{cc}((\hbX^t)_u, (\hbX^t)_s)$ for all $\hbX_{[0,t]} \in \Lambda_T^\alpha$ and some constant $C_{N,\alpha} > 0$ (see \cite[Theorem~9.5]{Friz2010} for the $p$-variation case, which carries over to the $\alpha$-H{\"o}lder setting by \cite[p.~182]{Friz2010}), we further obtain
  \begin{align}\label{EqThmUATHoelderProof3}
    \vert a(\hbX_{[0,t]}) \vert
    &\leq C_1 \Bigl( d_{cc}(\hbbX^N_t, \widehat{\mathbb{X}}^N_0)^N+2 \Bigr) \leq C_1 \Bigl( T^{\alpha N} \Bigl(\sup_{u,s \in [0,T],\, u < s} \frac{d_{cc}(\hbbX^N_u, \hbbX^N_s)}{\vert s-u \vert^\alpha}\Bigr)^N + 2 \Bigr) \nonumber\\
    &\leq C_1 \Bigl( C^N_{N,\alpha} T^{\alpha N}\Bigl( \sup_{u,s \in [0,T],\, u < s} \frac{d_{cc}((\hbX^t)_u, (\hbX^t)_s)}{\vert s-u \vert^\alpha}\Bigr)^N + 2 \Bigr)\nonumber\\
    &= C_1 \Bigl( C^N_{N,\alpha} T^{\alpha N} \Vert \hbX_{[0,T]}^t \Vert_{cc,\alpha}^N + 2 \Bigr).
  \end{align}

  Thus, we conclude that,
  \begin{equation*}
    \lim_{R \rightarrow \infty} \sup_{\hbX_{[0,t]}\in \Lambda_T^\alpha \setminus K_R} \frac{\vert a(\hbX_{[0,t]}) \vert}{\psi(\hbX_{[0,t]})} \leq C_1 \lim_{R \rightarrow \infty} \sup_{\hbX_{[0,t]} \in \Lambda_T^\alpha \setminus K_R} \frac{C^N_{N,\alpha} T^{\alpha N} \Vert \hbX_{[0,T]}^t \Vert^N_{cc,\alpha} + 2}{\exp\Bigl( \beta \Vert \hbX_{[0,T]}^t \Vert_{cc,\alpha}^{\gamma} \Bigr)} = 0,
  \end{equation*}
  since the exponential function dominates any polynomial. It follows from Lemma~\cite[Lemma~2.7]{Cuchiero2024} that $a \in \cB_\psi(\Lambda_T^\alpha)$, which shows that $\cA \subseteq \cB_\psi(\Lambda_T^\alpha)$.

  Moreover, we observe that $\cA$ is by the shuffle property a subalgebra of $\cB_\psi(\Lambda_T^\alpha)$. In order to show that $\cA$ is point separating and nowhere vanishing of $\psi$-moderate growth, we claim that the vector subspace $\widetilde{\cA} \subseteq \cA$ defined in \eqref{EqThmUATHoelderProof1} is point separating, nowhere vanishing, and for every $\tilde{a} \in \widetilde{\cA}$ there exists some $\lambda > 0$ such that $\exp( \lambda \vert \tilde{a}(\cdot) \vert) \in \mathcal{B}_\psi(\Lambda_T^\alpha)$.

  For the former, let $\hbY_{[0,t]}, \widehat{\mathbf{Z}}_{[0,s]} \in \Lambda_T^\alpha$ be distinct. If $t\neq s$, the two components are separated by the time component, i.e.~$\langle e_0,\hbbY_t\rangle\neq \langle e_0,\widehat{\mathbb{Z}}_s\rangle$. Thus, assume from now on that $t=s$. By contradiction, let us assume that for every $k \in \N_0$, $N \in \lbrace 0,\ldots,\lfloor 1/\alpha \rfloor \rbrace$, and $I \in \lbrace 0,\ldots,d \rbrace^N$ it holds that
  \begin{equation*}
    \langle ( e_I \shuffle e_0^{\otimes k} ) \otimes e_0, \hbbY_t \rangle = \langle ( e_I \shuffle e_0^{\otimes k} ) \otimes e_0, \widehat{\mathbb{Z}}_t \rangle,
  \end{equation*}
  where we observe, using the shuffle property, that
  \begin{equation}\label{eq:formAtilde}
    \langle ( e_I \shuffle e_0^{\otimes k} ) \otimes e_0, \hbbX_t \rangle = \int_0^t \langle e_I \shuffle e_0^{\otimes k}, \hbbX_r \rangle \dd r = \int_0^t \langle e_I, \hbbX_r \rangle \langle e_0^{\otimes k}, \hbbX_r \rangle \dd r = \int_0^t \langle e_I, \hbbX_r \rangle \frac{r^k}{k!} \dd r,
  \end{equation}
  for all $\hbX_{[0,t]} \in \Lambda_T^\alpha$. Thus, we conclude for every $k \in \N_0$, $N \in \lbrace 0,\ldots,\lfloor 1/\alpha \rfloor \rbrace$, and $I \in \lbrace 0,\ldots,d \rbrace^N$ that
  \begin{equation*}
    \int_0^t \langle e_I, \hbbY_r - \widehat{\mathbb{Z}}_r \rangle \frac{r^k}{k!} \dd r = 0.
  \end{equation*}
  By \cite[Corollary~4.24]{Brezis2011}, we then deduce that
  \begin{equation*}
    \langle e_I, \hbbY_r \rangle = \langle e_I, \widehat{\mathbb Z}_r \rangle,
  \end{equation*}
  for all $r \in [0,t]$ and all $I \in \{0, \ldots,d\}^N$, $N \in \{0,1,\ldots,\lfloor 1/\alpha\rfloor\}$. This contradicts our assumption that $\widehat{\mathbf Y}_{[0,t]}$ and $\widehat{\mathbf Z}_{[0,s]}$ are distinct, and shows that $\widetilde{\mathcal A}$ is point separating.

  Further, we observe that $\widetilde{\cA}$ vanishes nowhere. Indeed, by using the map
  \begin{equation*}
    ( \hbX_{[0,t]} \mapsto \tilde{a}(\hbX_{[0,t]}) := \langle e_\emptyset, \hbbX_t \rangle +\langle (e_\emptyset\shuffle e_0^{\otimes 0})\otimes e_0,\hbbX_t\rangle ) \in \widetilde{\cA},
  \end{equation*}
  we observe that $\tilde{a}(\hbX_{[0,t]}) = 1+\int_0^t \dd s=1+t\neq 0$, for all $\hbX_{[0,t]} \in \Lambda_T^\alpha$.

  Now, to show that for every $\tilde a\in\widetilde\cA$ there exists some $\lambda>0$ such that $\exp(\lambda|\tilde a(\cdot)|)\in\cB_\psi(\Lambda_T^\alpha)$ we fix some $( \hbX_{[0,t]} \mapsto \tilde{a}(\hbX_{[0,t]}) = l(\hbbX_t) ) \in \widetilde{\cA}$ with linear function
  \begin{equation*}
    l(\hbbX_t) = a_\emptyset \langle e_\emptyset,\hbbX_t\rangle + \sum_{0 \leq \vert I \vert \leq N} \sum_{k=0}^K a_{I,k} \langle ( e_I \shuffle e_0^{\otimes k} ) \otimes e_0, \widehat{\mathbb{X}}_t \rangle,
  \end{equation*}
  for some $K \in \N_0$ and $N \in \lbrace 0,\ldots,\lfloor 1/\alpha \rfloor \rbrace$ and $a_{I,k},a_\emptyset \in \R$. Then, by similar arguments as for \eqref{eq: cts}, we have $\exp( \vert \lambda \widetilde{a}(\cdot) \vert )\vert_{K_R} \in C(K_R)$, for all $\lambda, R > 0$. In addition, by the same reasoning as in \eqref{EqThmUATHoelderProof3}, together with the explicit form of the elements of $\widetilde{\cA}$ in \eqref{eq:formAtilde}, we deduce for all $\hbX_{[0,t]} \in \Lambda_T^\alpha$ that
  \begin{equation*}
    \begin{aligned}
      \vert \tilde{a}(\hbX_{[0,t]}) \vert & = \vert l(\hbbX_t) \vert \leq C_1 \Vert l \Vert_{T^{N+K+1}(\R^{d+1})^*} \Bigl( T^{\alpha (K+1) N} \sup_{u,s \in [0,T], \, u < s} \Bigl(\frac{d_{cc}(\hbbX^{N}_u,\hbbX^{N}_s)}{\vert s-u \vert^\alpha}\Bigr)^N +1\Bigr) \\
      & \leq C_1 \Vert l \Vert_{T^{N+K+1}(\R^{d+1})^*}\Bigl( C^N_{N,\alpha}  T^{\alpha (K+1) N} \Bigl(\sup_{u,s \in [0,T], \, u < s} \frac{d_{cc}((\hbX^t)_u,(\hbX^t)_s)}{\vert s-u \vert^\alpha}\Bigr)^N +1 \Bigr) \\
      & = C_1 \Vert l \Vert_{T^{N+K+1}(\R^{d+1})^*} \Bigl( C^N_{N,\alpha}  T^{\alpha (K+1)N} \Vert \hbX^t_{[0,T]} \Vert_{cc,\alpha}^N +1 \Bigr).
    \end{aligned}
  \end{equation*}
  Then, for $C_2 := \max(C_1 \Vert l \Vert_{T^{N+K+1}(\R^{d+1})^*} C^N_{N,\alpha} T^{\alpha(K+1)N}, C_1 \Vert l \Vert_{T^{N+K+1}(\R^{d+1})^*}) > 0$, we have
  \begin{equation*}
    \lim_{R \rightarrow \infty} \sup_{\hbX_{[0,t]} \in \Lambda_T^\alpha \setminus K_R} \frac{\exp( \lambda \vert \tilde{a}(\hbX_{[0,t]}) \vert)}{\psi(\hbX_{[0,t]})} \leq \lim_{R \rightarrow \infty} \sup_{\hbX_{[0,t]} \in \Lambda_T^\alpha \setminus K_R} \frac{\exp( \lambda C_2 (\| \hbX_{[0,T]}^t \|_{cc,\alpha}^N +1) )}{\exp( \beta \Vert \hbX_{[0,T]}^t \Vert_{cc,\alpha}^\gamma )} = 0,
  \end{equation*}
  where the last equality follows by choosing $\lambda < \beta/C_2$ small enough ensuring that the denominator tends faster to infinity than the numerator (as $\gamma \geq \lfloor 1/\alpha \rfloor \geq N$). Hence, by \cite[Lemma~2.7]{Cuchiero2024} it follows that $\exp( \lambda \vert \tilde{a}(\cdot) \vert) \in \cB_\psi(\Lambda_T^\alpha)$ which holds true for any $\tilde{a} \in \widetilde{\cA}$.
  
  Hence, we can apply the weighted real-valued Stone--Weierstrass theorem to conclude that $\cA$ is dense in $\cB_\psi(\Lambda_T^\alpha)$.
\end{proof}

\begin{remark}
  A related universal approximation result on weighted spaces is established in \cite[Theorem~2.20]{Cuchiero2024b}. There, the authors work with stopped Stratonovich-enhanced continuous semimartingales and their signatures. This corresponds to a specific subclass of stopped geometric rough paths, naturally lying in the rough path regularity range $\alpha\in(\frac{1}{3},\frac{1}{2})$. In contrast, Proposition~3.11 is formulated on the general stopped geometric $\alpha$-H\"older rough path space $\Lambda_T^\alpha$, for arbitrary $\alpha\in(0,1)$. Thus, while \cite[Theorem~2.20]{Cuchiero2024b} is tailored to the semimartingale setting, Proposition~3.11 provides an abstract rough path version which is not restricted to rough paths arising from semimartingales, e.g., it is applicable to Gaussian processes, like the fractional Brownian motion, see Section~\ref{sec: gaussian process}.
\end{remark}

We are now in a position to formulate a global universal approximation theorem in a suitable $L^p(\Lambda_T^\alpha)$-space. For this purpose, we work on the space $(\Lambda_T^\alpha,\cB(\Lambda_T^\alpha))$ equipped with a finite Borel measure $\nu$, where $\cB(\Lambda_T^\alpha)$ denotes the Borel $\sigma$-algebra on $\Lambda_T^\alpha$.

\begin{theorem}[$L^p$-universal approximation theorem on $\Lambda_T^\alpha$]\label{thm:Lpmain 2}
  Let $\psi$ be defined as in \eqref{eq:weight function non-anticipative}, $p\ge 1$, and $\int_{\Lambda_T^\alpha}\psi^p\dd\nu<\infty$. Moreover, consider the set
  \begin{equation*}
    \mathcal L_\Lambda:=\Bigl\{f_\ell|~f_\ell\colon\hbX_{[0,t]}\mapsto\boldsymbol\ell(\hbbX_t)=\sum_{|I|\le N}\ell_I\langle e_I,\hbbX_t\rangle,\ell_I\in\R,\,N\in\N_0,\,\hbX_{[0,t]}\in\Lambda_T^\alpha\Bigr\}.
  \end{equation*}
  Then, for every $f\in L^p(\Lambda_T^\alpha)$  and for every $\epsilon>0$ there exists a functional $f_\ell\in\mathcal L_\Lambda$ such that
  \begin{equation*}
    \|f-f_\ell\|_{L^p(\Lambda_T^\alpha)}<\epsilon.
  \end{equation*}
\end{theorem}

\begin{proof}
  We apply Lusin's theorem and Tietze's extension theorem verbatim as in Theorem~\ref{thm:Lpmain}, and we obtain that for every $f\in L^p(\Lambda_T^\alpha,\nu)$ and every $\epsilon>0$, there exist $K^\epsilon>0$ and a bounded continuous function $f^\varepsilon\in C_b(\Lambda_T^\alpha;[-K^\epsilon,K^\epsilon])$ with $\|f-f^\epsilon\|_{L^p(\Lambda_T^\alpha)}<\frac{\epsilon}{2}$.

  By definition $C_b(\Lambda_T^\alpha)\subseteq\mathcal B_\psi(\Lambda_T^\alpha)$ and, using Proposition~\ref{prop: weighted UAT 2}, we can approximate $f^\epsilon$ in $\mathcal B_\psi(\Lambda_T^\alpha)$ by a linear function on the signature, i.e.
  \begin{equation*}
    \|{f}^\epsilon-f_\ell\|^p_{\cB_\psi(\Lambda_T^\alpha)}=\Bigl(\sup_{\hbX_{[0,t]}\in\Lambda_T^\alpha}\frac{|{f}^\epsilon(\hbX_{[0,t]})-\boldsymbol{\ell}(\hbbX_t)|}{\psi(\hbX_{[0,t]})}
    \Bigr)^p< \frac{\epsilon^p}{2^p M},
  \end{equation*}
  where $M:=\int_{\Lambda_T^\alpha}\psi^p\dd\nu<\infty$. As in Proposition~\ref{prop: weighted UAT 2}, this yields an $L^p$-approximation of $f$ by such linear combinations, that is,
  \begin{equation*}
    \|{f}^\epsilon-f_\ell\|^p_{L^p(\Lambda_T^\alpha)}\le\int_{\Lambda_T^\alpha}\psi^p\dd\nu~\|{f}^\epsilon-f_\ell\|^p_{\cB_\psi(\Lambda_T^\alpha)}< \Bigl(\frac{\epsilon}{2} \Bigr)^p,
  \end{equation*}
  which proves the claim.
\end{proof}

\begin{remark}\label{rem: UAT for robust signatures}
  In contrast to the classical signature employed in Theorem~\ref{thm:Lpmain 2}, the $L^p$-universal approximation theorems in \cite{Schell2023} and \cite{Bayer2025} are established using so-called robust signatures, which were introduced in \cite{Chevyrev2022} as a normalized variant of the classical signature. Moreover, the approaches developed in \cite{Schell2023} and \cite{Bayer2025} differ substantially from the proof of Theorem~\ref{thm:Lpmain 2}.

  More specifically, \cite{Schell2023} exploits that linear functionals of the bounded signature form a rich algebra of measurable functions that generates the $\sigma$-algebra of the underlying (subsets of the) classical path space; a monotone class argument then yields $L^2$-density of linear signature functionals among all square-integrable measurable random variables. By contrast, \cite{Bayer2025} reduces the approximation of general $L^p$-functionals to that of bounded continuous ones and combines suitable weight functions --- used to control the tail behavior of the underlying measure on the rough path space --- with a Stone--Weierstrass theorem for robust signatures.
\end{remark}

\section{Approximation properties of linear functionals on Gaussian signatures}\label{sec: gaussian process}

In this section, we demonstrate that the $L^p$-universal approximation theorems (Theorem~\ref{thm:Lpmain} and Theorem~\ref{thm:Lpmain 2}) apply to a class of (time-extended) Gaussian rough paths whose covariance functions satisfy suitable variation estimates. This allows us to approximate fairly general stochastic processes by linear combinations of the random signatures of time-extended Gaussian processes. In particular, we show that the framework applies to fractional Brownian motion with Hurst parameter $H\in(\frac{1}{3},\frac{1}{2}]$ and, thus, especially to the standard Brownian motion as the special case $H=\frac12$.

The deterministic assumptions needed for these theorems, such as the weighted-space property and the density of linear signature functionals on the corresponding rough path spaces, have been verified in Section~\ref{sec: global approx}. Hence, for the law of the Gaussian rough path, it remains to verify the integrability condition required in Theorem~\ref{thm:Lpmain} and Theorem~\ref{thm:Lpmain 2}, namely the corresponding exponential moment condition for the law of a Gaussian rough path. Establishing this condition constitutes the main step of the proofs below. For related approximation result for stochastic processes using the robust signature, we refer to \cite{Schell2023,Bayer2025}.

\medskip

Throughout the present section, let $X=(X_t)_{t\in [0,T]}$ be a $d$-dimensional continuous, centred Gaussian process with independent components defined on a probability space $(\Omega, \cF, \P)$, with a filtration $(\cF_t)_{t \in [0,T]}$ satisfying the usual conditions, i.e., completeness and right-continuity. For an introduction to stochastic processes and stochastic calculus, we refer, e.g., to the classical textbook \cite{Karatzas1988}.

Let us recall a sufficient condition under which a Gaussian process admits a geometric rough path lift, see \cite[Theorem~10.4]{Friz2020}. To that end, we denote by
\begin{align*}
  R\colon [0,T]^2 &\to\R^{d\times d},\\
  (s,t) &\mapsto \E[X_{s}\otimes X_{t}],
\end{align*}
the covariance function of $X$.  For the $i$-th component of $X$, we write $R_{X^i}(s,t):=\E[X_s^iX_t^i]$, $0\le s,t\le T.$ Assume that there exist $\varrho\in[1,2)$ and $M<\infty$ such that, for every $i\in\{1,\dots,d\}$ and all $0\leq s\leq t\leq T$,
\begin{equation*} 
  \|R_{X^i}\|_{\varrho;[s,t]^2}\le M |t-s|^{1/\varrho},
\end{equation*}
where $\|\,\cdot\,\|_{\varrho\textup{-var};I\times I^\prime}$ denotes the two-dimensional $\varrho$-variation on a rectangle $I\times I^\prime$ as defined in \cite[(10.5)]{Friz2020}. Then, for $1\le i<j\leq d$ and $0\le s\le t\le T$, the second-level iterated integrals are defined in the $L^2$-sense by
\begin{equation*}
  \X^{(2),i,j}_{s,t}:=\lim_{|\mathcal P|\to 0}\int_{\mathcal P} (X^i_r-X^i_s)\dd X^j_r:=\lim_{|\mathcal P|\to 0}\sum_{[u,v]\in\mathcal P}
  (X_u^i-X_s^i)X^j_{u,v},
\end{equation*}
where $\mathcal P$ is a finite partition of $[s,t]$ and $|\mathcal P|$ denotes its mesh size. The remaining components are determined by the algebraic relations
\begin{equation*}
  \X^{(2),i,i}_{s,t}:=\frac{1}{2} (X^i_{s,t})^2,
\end{equation*}
and, for $1\le i<j\le d$,
\begin{equation*}
  \X^{(2),j,i}_{s,t}:=-\X^{(2),i,j}_{s,t}+X^i_{s,t}X^j_{s,t}.
\end{equation*}
Then, for $\varrho\in[1,\frac{3}{2})$ and any $\alpha\in(\frac{1}{3},\frac{1}{2\varrho})$ the Gaussian rough path $\bX=(1,X,\X^{(2)})$ belongs to $C^{0,\alpha}([0,T];G^2(\R^d))$ almost surely and, thus, $X$ admits a geometric $\alpha$-H\"older rough path lift. As usual, we denote by $\hbX$ the geometric rough path of the time-extended Gaussian process $\hX=(\cdot,X)$ and $\hbbX$ is associated signature, which by definition of the signature of a geometric rough path, corresponds to the unique Lyons' lift of $\hbX$. We call $\hbbX$ the time-extended Gaussian signature.
 
Furthermore, we introduce the filtration $\cF_t^{\bX}:=\sigma(\{\bX_{s}: s\le t\},\mathcal{N})$ for $t\in [0,T]$ and $\mathcal{N}$ containing all $\P$-null sets, i.e., the natural augmented filtration generated by $\bX$. We denote by $\mathcal H^p$ the space of $(\mathcal F^{\bX}_t)$-progressively measurable processes $A$ such that
\begin{equation*}
  \|A\|_{\mathcal H^p}^p:=\E\Bigl[\int_0^T|A_t|^p\dd t\Bigr]<\infty.
\end{equation*}
  In addition, we define the $\alpha$-H{\"o}lder rough path norm
  \begin{equation*}
    \ver\hbX_\alpha:=\|\widehat X\|_\alpha+\sqrt{\|\hbbX^{(2)}\|_{2\alpha}}=\sup_{0\le s<t\le T}\frac{|\widehat X_{s,t}|}{|t-s|^\alpha}+\sqrt{\sup_{0\le s<t\le T}\frac{|\hbbX^{(2)}_{s,t}|}{|t-s|^{2\alpha}}},
  \end{equation*}
  for $\hbX\in C_0^\alpha([0,T];G^2(\R^{d+1}))$ and $\alpha\in(\frac{1}{3},\frac{1}{2\varrho})$. Note that this norm is equivalent to the norm $\|\hbX\|_{cc,\alpha}$ on $G^2(\R^{d+1})$ (with constant $C>0$), see \cite[p.22]{Friz2020}.
In Section~\ref{sec: global approx} we introduced the notion of a stopped rough path in general, see Definition~\ref{def: stopped rough path}. We now specialise this construction to the time-extended Gaussian rough path and present an explicit description of its coordinates.

\begin{example}
  By Definition~\ref{def: stopped rough path} the stopped Gaussian rough path $\hbX^t_{[0,T]}$ is given by $(\hbX^t)_s:=\hbX_s$ for all $s\in[0,t]$ and for all $r\in[t,T]$ we have
  \begin{equation*}
    \langle e_I,(\hbX^t)_r\rangle=
    \begin{cases}
    r,&  \text{for } I=(0)\\[10pt]
    \frac{1}{2} r^{2},& \text{for } I=(0,0)\\[10pt]
    \langle e_I, \hbX_t \rangle,& \text{for } I=(i)\text{ or } I=(j,i), i\in\{1,\ldots,d\},\\
    & j\in\{0,\ldots,d\}\\[10pt]
    r \cdot \langle e_{i}, \hbX_t \rangle
    - \langle e_{(0,i)}, \hbX_t \rangle,& \text{for } I=(i,0), i\in\{1,\ldots,d\},
    \end{cases}
  \end{equation*}
  where the last line follows by
  \begin{align*}
    \langle e_{(i,0)},(\hbX^t)_r\rangle&=\int_0^r\langle e_i,\hbX^t_s\rangle\dd s\\
    &=\int_0^t \langle e_i,\hbX^t_s\rangle\dd s+\int_t^r\langle e_i,\hbX_t\rangle\dd s\\
    &=\langle e_{(i,0)},\hbX_t\rangle+(r-t)\langle e_i,\hbX_t\rangle\\
    &=\langle e_{(i,0)},\hbX_t\rangle+r\langle e_i,\hbX_t\rangle-\langle e_0,\hbX_t\rangle\langle e_i,\hbX_t\rangle\\
    &=\langle e_{(i,0)},\hbX_t\rangle+r\langle e_i,\hbX_t\rangle-\langle e_0\shuffle e_i,\hbX_t\rangle\\
    &=\langle e_{(i,0)},\hbX_t\rangle+r\langle e_i,\hbX_t\rangle-\langle e_{(0,i)},\hbX_t\rangle-\langle e_{(i,0)},\hbX_t\rangle\\
    &=r\langle e_i,\hbX_t\rangle-\langle e_{(0,i)},\hbX_t\rangle.
  \end{align*}
\end{example}

\subsection{Universal approximation with Gaussian signatures}

In this subsection, we establish that any functional $f(\hbX) \in L^p(\Omega,\P)$, as well as any stochastic process $f(\hbX_{[0,\cdot]}) \in \mathcal H^p$, can be approximated by linear functionals acting on the (time-extended) signature of Gaussian processes.

\begin{proposition}\label{prop: approximation Gaussian signature}
  Let $X$ be a $d$-dimensional centred Gaussian process with independent components and covariance $R$ such that there exists $\varrho\in[1,\frac{3}{2})$ and $M<\infty$ such that for every $i\in\{1,\ldots,d\}$ and $0\le s\le t\le T$,
  \begin{equation*}
    \|R_{X^i}\|_{\varrho\textup{-var};[s,t]^2}\le M|t-s|^{1/\varrho}.
  \end{equation*}  
  Let $\alpha \in (\frac{1}{3},\frac{1}{2\varrho})$ and let $\hX=(\cdot,X)$ be the time-extended Gaussian process and $\hbX$ the corresponding time-extended Gaussian rough path.
  \begin{enumerate}
    \item[(i)] Let $f(\hbX)\in L^p(\Omega;\P)$ with $f\colon \widehat{C}_{d,T}^{\alpha} \to \R$. Then, for every $\epsilon>0$ there exists a linear function $\boldsymbol\ell \colon T((\R^{d+1}))\to \R$ of the form $\hbbX_T\mapsto\boldsymbol\ell(\hbbX_T):=\sum_{|I|\le N}\ell_I\langle e_I,\hbbX_T\rangle$, for some $N\in\N_0$ and $\ell_I\in\R$, such that
    \begin{equation*}
      \E[|f(\hbX)-\boldsymbol\ell(\hbbX_T)|^p]<\epsilon.
    \end{equation*}
    \item[(ii)] Let $f(\hbX_{[0,\cdot]})\in\mathcal H^p$ with $f\colon \Lambda_T^{\alpha}\to \R$. Then, for every $\epsilon>0$ there exists a linear function $\boldsymbol\ell\colon T((\R^{d+1}))\to \R$ of the form $\hbbX_t\mapsto\boldsymbol\ell(\hbbX_t):=\sum_{|I|\le N}\ell_I\langle e_I,\hbbX_t\rangle$, for some $N\in\N_0$ and $\ell_I\in\R$, such that
    \begin{equation*}
      \E\Bigl[\int_0^T|f(\hbX_{[0,t]})-\boldsymbol\ell(\hbbX_t)|^p\dd t\Bigr]<\epsilon.
    \end{equation*}
  \end{enumerate}
\end{proposition}

\begin{proof}
  \emph{(i):} As discussed in \cite[Appendix~A.1]{Friz2010}, the Gaussian rough path $\hbX$ can be seen as a $C^{0,\alpha}([0,T];G^{2}(\R^{d+1}))$-valued random variable and its law $\mu_{\hbX}$ is a Borel probability measure on $C^{0,\alpha}([0,T];G^2(\R^{d+1}))$. Thus, when working on the space $\widehat{C}_{d,T}^\alpha$ of time-extended geometric rough paths, we take $\nu := \mu_{\hbX}$. Then, we observe that since $f(\hbX)\in L^p(\Omega;\P)$, we have that
  \begin{align*}
    \int_{\widehat{C}_{d,T}^\alpha}|f|^p\dd\mu_\hbX=\E[|f(\hbX)|^p]<\infty,
  \end{align*}
  that is, $f\in L^p(\widehat{C}_{d,T}^\alpha;\mu_\hbX)$.

In order to apply Theorem~\ref{thm:Lpmain}, we have to verify that the time-extended Gaussian rough path $\hbX$ satisfies the exponential moment condition given by $\int_{\widehat{C}_{d,T}^\alpha}\psi^p\dd\nu<\infty$, with $\psi(\hbX)=\exp(\beta p \|\hbX\|_{cc,\alpha}^\gamma)$ for $\gamma\ge \lfloor 1/\alpha\rfloor$, $\beta>0$, and $\alpha\in(\frac{1}{2},\frac{1}{2\varrho})$.

 The Fernique theorem for Gaussian rough paths \cite[Theorem~11.9]{Friz2020} yields that for any $\alpha\in(\frac{1}{3},\frac{1}{2\varrho})$ there exists $\eta=\eta(M,T,\alpha,\varrho)>0$ with
  \begin{equation*}
    \E\Bigl[\exp\Bigl(\eta\ver{\hbX}_\alpha^2\Bigr)\Bigr]<\infty.
  \end{equation*}
Hence, for $\gamma=2$, using the equivalence of the norm $\|\,\cdot\,\|_{cc,\alpha}$ and the norm $\ver{\,\cdot\,}_\alpha$ on $G^2$, we obtain
 \begin{equation*}
    \int_{\widehat{C}_{d,T}^\alpha}\psi^p\dd\mu_{\hbX}=\E\Bigl[\exp\Bigl(\beta p\| \hbX\|_{cc,\alpha}^2\Bigr)\Bigr]\le \E\Bigl[\exp\Bigl(\beta p\, C^2\ver \hbX_\alpha^2\Bigr)\Bigr]<\infty,
 \end{equation*}
 whenever
  \begin{equation*}
    0<\beta\le \frac{\eta}{pC^2}.
  \end{equation*}
 With this choice of $\beta$, the integrability condition required in Theorem~\ref{thm:Lpmain} is satisfied. Therefore, Theorem~\ref{thm:Lpmain} yields that for every $\epsilon>0$ there exists a functional $f_\ell\in\mathcal L$ such that
  \begin{equation*}
    \|f-f_\ell\|^p_{L^p(\widehat{C}_{d,T}^\alpha)}<\epsilon.
  \end{equation*}
  In particular, this implies that, for every $\epsilon>0$ there exists a linear function $\boldsymbol\ell$ on the Gaussian signature, such that
  \begin{equation*}
    \E[|f(\hbX)-\boldsymbol\ell(\hbbX_T)|^p]=\int_{\widehat{C}_{d,T}^\alpha}|f(\hbX)-f_\ell(\hbX)|^p\dd\mu_\hbX=\|f-f_\ell\|^p_{L^p(\widehat{C}_{d,T}^\alpha)}<\epsilon.
  \end{equation*}

  \emph{(ii):} On the space $(\Lambda_T^\alpha,\cB(\Lambda_T^\alpha))$, we let $\nu$ be the push-forward measure of $\dd t\otimes \dd\mu_{\hbX}$ under the surjective map
  \begin{equation*}
    \phi\colon [0,T]\times \widehat{C}_{d,T}^\alpha\to \Lambda_T^\alpha,\quad (t,\hbX)\mapsto \hbX_{[0,t]},
  \end{equation*}
  that is, $\nu := (\dd t \otimes \dd\mu_{\hbX})\circ \phi^{-1}$.

  We first show that $f(\hbX_{[0,\cdot]})\in L^p(\Lambda_T^\alpha)$. By a change of measure result, we have
  \begin{align*}
    \|f\|^p_{L^p(\Lambda_T^\alpha)}&=\int_{\Lambda_T^\alpha}|f|^p\dd\nu\\
    &=\int_{\widehat C_{d,T}^\alpha}\int_0^T|(f\circ \phi)(t,\hbX)|^p\dd t\dd\mu_{\hbX}\\
    &=\E\Bigl[\int_0^T|f(\hbX_{[0,t]})|^p\dd t\Bigr]<\infty,
  \end{align*}
  since $f(\hbX_{[0,\cdot]})\in \mathcal H^p$ by assumption. Next, we verify the exponential moment condition as required in Theorem~\ref{thm:Lpmain 2}. By a change of measure result, we get
  \begin{align*}
    \int_{\Lambda_T^\alpha}\psi^p\dd\nu&=\int_{\widehat{C}_{d,T}^\alpha}\int_0^T((\psi\circ\phi)(t,\hbX))^p\dd t\dd\mu_{\hbX}\\
    &=\E\Bigl[\int_0^T\psi(\hbX_{[0,t]})^p\dd t\Bigr]\\
    &=\E\Bigl[\int_0^T\exp\Bigl(\beta p\| \hbX^t_{[0,T]}\|_{cc,\alpha}^\gamma\Bigr)\dd t\Bigr]\\
    &\le T\E\Bigl[\sup_{t\in[0,T]}\exp\Bigl(\beta p\| \hbX^t_{[0,T]}\|_{cc,\alpha}^\gamma\Bigr)\Bigr]\\
    &=T\E\Bigl[\exp\Bigl(\beta p\| \hbX_{[0,T]}\|_{cc,\alpha}^\gamma\Bigr)\Bigr]\\
    & \le T \E\Bigl[\exp\Bigl(\beta p\, C^\gamma\ver \hbX_\alpha^\gamma\Bigr)\Bigr]<\infty,
  \end{align*}
  for $\gamma=2$ and $\beta\in(0,\frac{\eta}{C^\gamma p}]$, where we used that
  \begin{equation*}
    \sup_{t\in[0,T]}\|\hbX^t_{[0,T]}\|_{cc,\alpha}=\|\hbX_{[0,T]}\|_{cc,\alpha}.
  \end{equation*}

  Therefore, by Theorem~\ref{thm:Lpmain 2} for every $\epsilon>0$ there exists a functional $f_\ell\in\mathcal L_\Lambda$ such that
  \begin{equation*}
    \|f-f_\ell\|^p_{L^p(\Lambda_T^\alpha)}<\epsilon.
  \end{equation*}
  Consequently, for every $\epsilon>0$ there exists a linear function $\boldsymbol\ell$ on the Gaussian signature, such that
  \begin{align*}
    \E\Bigl[\int_0^T |f(\hbX_{[0,t]})-\boldsymbol\ell(\hbbX_t)|^p\dd t\Bigr]&=\int_{\widehat C_{d,T}^\alpha}\int_0^T|(f\circ \phi-f_\ell\circ\phi)(t,\hbX)|^p\dd t\dd\mu_{\hbX}\\
    &=\int_{\Lambda_T^\alpha}|f-f_\ell|^p\dd\nu\\
    &=\|f-f_\ell\|^p_{L^p(\Lambda_T^\alpha)}<\epsilon,
  \end{align*}
  where again we used a change of measure result. This concludes the proof.
\end{proof}
\subsubsection{Application to fractional Brownian motion and affine transformations}

We now discuss some important examples to which the preceding Proposition~\ref{prop: approximation Gaussian signature} applies.

\begin{corollary}\label{cor: fBM}
  Let $W^H$ be a $d$-dimensional fractional Brownian motion with Hurst index $H\in(\frac{1}{3},\frac{1}{2}]$ and let $\alpha\in(\frac13,H)$. Then, there exists $\eta=\eta(T,\alpha,H)$ such that
  \begin{equation*}
    \E[\exp(\eta\ver{\hbW^H}_\alpha^2)]<\infty.
  \end{equation*}
\end{corollary}

\begin{proof}
  Let $H\in(\frac{1}{3},\frac{1}{2}]$. Then the covariance of the $d$-dimensional fractional Brownian motion $W^H$ has finite $\varrho$-variation with $\varrho=\frac{1}{2H}\in[1,\frac{3}{2})$, see \cite[Theorem~10.9 and Example~10.11]{Friz2020}. Hence, by \cite[Theorem~11.9]{Friz2020} there exists $\eta=\eta(T,\alpha,H)$ such that
  \begin{equation*}
    \E\Bigl[\exp\Bigl(\eta\ver{\hbW^H}_\alpha^2\Bigr)\Bigr]<\infty.
  \end{equation*}
\end{proof}

Consequently, the $L^p$-universal approximation results of Theorems~\ref{thm:Lpmain} and Theorem~\ref{thm:Lpmain 2} apply to the law of $\hbW^H$, with $\gamma=2$ and $\beta>0$ chosen sufficiently small. In particular, for $H=\frac{1}{2}$ this results corresponds to stanadard Brownian motion.

\begin{corollary}
  Let $X$ be a $d$-dimensional centred Gaussian process with independent components and covariance $R$ such that there exists $\varrho\in[1,\frac{3}{2})$ and $M<\infty$ such that for every $i\in\{1,\ldots,d\}$ and $0\le s\le t\le T$,
  \begin{equation*}
    \|R_{X^i}\|_{\varrho\textup{-var};[s,t]^2}\le M|t-s|^{1/\varrho}.
  \end{equation*}
  Let $\alpha\in(\frac{1}{3},\frac{1}{2\varrho})$. Set $\hX_t=(t,X_t)\in\R^{d+1}$, $t\in[0,T]$ and let $\hbX$ denote its geometric rough path lift. Furthermore, let $Z$ be an affine transformation of $X$, i.e.
  \begin{equation*}
    Z_t=x_0+bt+\Sigma X_t,\qquad t\in[0,T],
  \end{equation*}
  where $x_0\in\R^d$, $b\in\R^d$, and $\Sigma\in\R^{d\times d}$. Set $\widehat{Z}_t:=(t,Z_t)\in\R^{d+1}$ and let $\widehat{\mathbf Z}$ be its  geometric rough path lift. Then, for every $p\ge 1$ there exists $\beta>0$ (depending on $p,\alpha,\varrho,T,b,\Sigma$) such that
  \begin{equation*}
    \E\Bigl[\exp\Bigl(\beta p\,\|\widehat{\mathbf Z}\|_{cc,\alpha}^2\Bigr)\Bigr]<\infty.
  \end{equation*}
\end{corollary}

\begin{proof}
  Define the linear map $A\colon\R^{d+1}\to\R^{d+1}$ by
  \begin{equation*}
    A(u,v):=(u,\,bu+\Sigma v),\qquad u\in\R,\ v\in\R^d.
  \end{equation*}
  Then, $\widehat Z_t=(0,x_0)+A\widehat{X}_t$ and, thus, $\widehat Z_{s,t}=A\widehat{X}_{s,t}$ for all $0\le s<t\le T$. The map $A$ induces a map $\Phi_A\colon G^2(\R^{d+1})\to G^2(\R^{d+1})$ given by
  \begin{equation*}
    \Phi_A(1,x^{(1)},x^{(2)})=\bigl(1,Ax^{(1)},A^{\otimes2}x^{(2)}\bigr),
  \end{equation*}
  where $A^{\otimes2}$ is defined by $A^{\otimes2}(u\otimes v)=(Au)\otimes(Av)$. Let $\widehat{\mathbf Z}=(1,\widehat Z,\widehat{\mathbb Z}^{(2)})$ and $\hbX=(1,\widehat{X},\hbbX^{(2)})$ denote the geometric rough path lifts of $\widehat Z$ and $\hX$, respectively, then
  \begin{equation*}
    \widehat{\mathbf Z}_{s,t}=\Phi_A(\hbX_{s,t}),\qquad 0\le s<t\le T,
  \end{equation*}
  i.e.,
  \begin{equation*}
    \widehat Z_{s,t}=A\hX_{s,t},\qquad \widehat{\mathbb Z}_{s,t}^{(2)}=A^{\otimes 2} \hbbX^{(2)}_{s,t}.
  \end{equation*}
  More precisely, if we identify $(\R^{k})^{\otimes 2}$ with $\R^{k\times k}$, then
  \begin{equation*}
    \widehat{\mathbb Z}_{s,t}^{(2)}=A \hbbX^{(2)}_{s,t} A^\top,
  \end{equation*}
  where $A^\top$ is the transpose of $A$. By the equivalence of $\|\,\cdot\,\|_{cc,\alpha}$ and the norm $\ver{\,\cdot\,}_\alpha$ on $G^2$, there exists a constant $C\ge 0$ such that
  \begin{align*}
    \|\widehat{\mathbf Z}\|_{cc,\alpha}&\le C \ver{\widehat{\mathbf Z}}_\alpha\\
    &=C(\|\widehat Z\|_\alpha+\sqrt{\|\widehat{\mathbb Z}^{(2)}\|_{2\alpha}})\\
    &=C(\|A\hX\|_\alpha+\sqrt{\|A^{\otimes 2}\hbbX^{(2)}\|_{2\alpha}})\\
    &\le C(\|A\|\|\hX\|_\alpha+\|A^{\otimes 2}\|^{1/2}\sqrt{\|\hbbX^{(2)}\|_{2\alpha}})\\
    &=C\|A\| \ver{\hbX}_\alpha,
  \end{align*}
  where $\|A\|$ denotes the operator norm induced by the Euclidean norm and $\|A^{\otimes2}\|^{1/2}$ denotes the operator norm of the linear map $A^{\otimes2}\colon(\R^{d+1})^{\otimes2}\to(\R^{d+1})^{\otimes2}$, which satisfy $\|A^{\otimes2}\|^{1/2}=\|A\|$.

  Finally, by Fernique's theorem (\cite[Theorem~11.9]{Friz2020}) there exists $\eta>0$ such that
  \begin{equation*}
    \E\Big[\exp\big(\eta\,\ver{\hbX}_\alpha^2\big)\Big]<\infty.
  \end{equation*}
  For $\beta\in (0,\eta/(C^2\|A\|^2p)]$, we then have
  \begin{equation*}
    \E\Big[\exp\big(\beta p\,\|\widehat{\mathbf Z}\|_{cc,\alpha}^2\big)\Big]\le \E\Big[\exp\big(\eta\,\ver{\hbX}_\alpha^2\big)\Big] <\infty,
  \end{equation*}
  which proves the claim.
\end{proof}

\subsection{Approximation of stochastic differential equations}

In this subsection, we show that solutions to stochastic differential equations (SDEs) driven by Brownian motions can be approximated by linear combinations of time-extended Brownian signatures.

\medskip

Therefore, throughout the present section, let $W=(W_t)_{t\in [0,T]}$ be a $d$-dimensional Brownian motion, defined on a probability space $(\Omega, \cF, \P)$, with a filtration $(\cF_t)_{t \in [0,T]}$ satisfying the usual conditions, i.e., completeness and right-continuity. 

\medskip

Recall that, for a Brownian motion~$W$, there is a canonical choice for a random geometric rough path lift $\bW$ of $W$ given by
\begin{equation*}
  \bW_t:=\bigg(1,W_t,\int_0^tW_s\otimes\circ \dd W_s\bigg),\quad t\in [0,T],
\end{equation*}
where the stochastic integral $\int_0^t W_s\otimes\circ \dd W_s$ is defined as a classical Stratonovich integral. Note that $\bW_t$ takes values in $ G^2(\R^{d})$ for all $t\in[0,T]$, see e.g. \cite[Exercise~13.10]{Friz2010}, and the Stratonovich-enhanced Brownian rough path  $\bW$ is, almost surely, a geometric $\alpha$-H{\"o}lder rough path for $\alpha\in(\frac{1}{3},\frac{1}{2})$. In the following, we denote the time-extended Stratonovich-enhanced Brownian rough path by $\hbW$ and $\hbbW$ its associated signature, which, by definition of the signature of a geometric rough path, corresponds to the unique Lyons' lift of $\hbW$ and coincides with iterated Stratonovich integrals, see \cite[Exercise~17.2]{Friz2010}. We call $\hbW$ and $\hbbW$ the (time-extended) Brownian rough path and the (time-extended) Brownian signature, respectively.

\begin{remark}
  Note that $\mathcal F_t^\bW=\sigma(\{\bW_s: s\le t\},\mathcal N)=\sigma(\{W_s: s\le t\},\mathcal N)=:\mathcal F_t^W$ for $t\in[0,T]$, that is, the natural augmented filtration generated by $\bW$ and by $W$ coincide, see e.g.~the proof of \cite[Proposition~13.11]{Friz2010}.
\end{remark}
\begin{proposition}\label{prop: SDE signature approx}
  Let $1\le p<\infty$. Consider the stochastic differential equation
  \begin{equation}\label{eq: sde}
    Y_t = y_0 + \int_0^t \mu(s,Y_s) \dd s + \int_0^t \sigma(s,Y_s) \dd W_s, \quad t \in [0,T],
  \end{equation}
  where $y_0 \in \R^m$, $\mu\colon [0,T]\times \R^m \to \R^m$ and $\sigma\colon [0,T] \times \R^m \to \R^{m\times d}$ are continuous functions, and $\int_0^t \sigma(s,Y_s) \dd W_s$ is defined as an It{\^o} integral. Suppose there exists a unique (strong) solution $Y$ to the SDE~\eqref{eq: sde} and that $\mu, \sigma$ satisfy the linear growth condition
  \begin{equation*}
    |\mu(t,x)|+|\sigma(t,x)|\leq C (1 + |x|),\quad x\in\R^m,
  \end{equation*}
  for some constant $C>0$.
 
  Then, for every $\epsilon>0$ there exists a linear function $\boldsymbol\ell\colon T((\R^{d+1}))\to \R^m$ of the form $\hbbW_t\mapsto\boldsymbol\ell(\hbbW_t):=\sum_{|I|\le N}\ell_I\langle e_I,\hbbW_t\rangle$, for some $N\in\N_0$ and $\ell_I\in\R^m$, such that
  \begin{equation*}
    \E\Bigl[\int_0^T|Y_t-\boldsymbol\ell(\hbbW_t)|^p\dd t\Bigr]<\epsilon.
  \end{equation*}
\end{proposition}

\begin{proof}
  We first clarify the three objects used in the proof. The process $Y$ denotes the unique strong solution of the original It\^o SDE~\eqref{eq: sde}. We then approximate the coefficients by smooth coefficients and denote by $Y^\epsilon$ the solution of the corresponding smooth It{\^o} SDE, driven by the same Brownian motion. Finally, after rewriting this smooth equation in Stratonovich form, $\Phi$ denotes the deterministic It{\^o}--Lyons solution map on the stopped rough path space.

  \textit{Step~1.} It is well-known that SDEs with coefficients satisfying a linear growth condition admit solutions that are uniformly bounded in $L^p(\Omega,\P)$, i.e.,
  \begin{equation*}
    \E\Bigl[\sup_{t\in[0,T]} |Y_t|^p\Bigr] < \infty,
  \end{equation*}
  see, for instance, the argument in \cite[Theorem~4.5.3]{Kloeden1992}.

  We next approximate the coefficients by smooth coefficients.  Following a similar construction as in the proof of \cite[Proposition~1.1]{Hofmanova2021}, we can find smooth functions $\mu^n$ and $\sigma^n$ with compact support such that the following two properties hold. First, there exists a constant $C_0>0$, independent of $n$, such that
  \begin{equation*}
    |\mu^n(t,y)|+|\sigma^n(t,y)| \leq C_0(1+|y|), \qquad (t,y)\in[0,T]\times\R^m .
  \end{equation*}
  Second, the approximating coefficients converge locally uniformly to $\mu$ and $\sigma$, namely for every $R>0$,
  \begin{equation}\label{eq: convergence coeff}
    \sup_{(t,y)\in[0,T]\times B_R}
    |\mu^n(t,y)-\mu(t,y)|+\sup_{(t,y)\in[0,T]\times B_R} |\sigma^n(t,y)-\sigma(t,y)| \to 0
  \end{equation}
  as $n\to\infty$, where $B_R:=\{y\in\R^m: |y|\le R\}.$ For each $n\in\N$, consider the approximating SDE
  \begin{equation*}
    Y_t^n=y_0 +\int_0^t \mu^n(s,Y_s^n)\dd s +\int_0^t \sigma^n(s,Y_s^n)\dd W_s, \qquad t\in[0,T].
  \end{equation*}

  Since $\mu^n$ and $\sigma^n$ are smooth with compact support, they are globally Lipschitz and bounded. Hence, for every $n$, the process $(Y_t^n)_{t\in [0,T]}$ admits a unique strong solution. By \cite[Theorem~A]{Kaneko1988}, together with \eqref{eq: convergence coeff} and pathwise uniqueness of the limiting equation \eqref{eq: sde}, we obtain
  \begin{equation*}
    \E[\sup_{t\in[0,T]} |Y_t^n-Y_t|^2] \to 0.
  \end{equation*}
  In particular, for $p=1$, we have
  \begin{equation*}
  \E[\sup_{t\in[0,T]}|Y_t^n-Y_t|]\le \E[\sup_{t\in[0,T]}|Y_t^n-Y_t|^2]^{\frac{1}{2}}\to 0.
  \end{equation*}
  Now, let $p>2$ be fixed. Choose $q>p$. Since $y_0$ is deterministic and the coefficients $\mu^n,\sigma^n$ satisfy the uniform linear growth bound, the usual Burkholder--Davis--Gundy and Gr{\"o}nwall estimates yield
  \begin{equation*}
    \sup_{n\ge1} \E\Bigl[ \sup_{t\in[0,T]} |Y_t^n|^q \Bigr] <\infty,\qquad \sup_{n\ge 1}\E[\sup_{t\in[0,T]}|Y_t^n-Y_t|^q]<\infty.
  \end{equation*}
  We interpolate between $L^2$ and $L^q$. For $\theta\in(0,1)$, let
  \begin{equation*}
    \frac{1}{p}=\frac{\theta}{2}+\frac{1-\theta}{q}.
  \end{equation*}
  Then, we obtain
  \begin{equation*}
    \|Y^n-Y\|_{L^p}\le \|Y^n-Y\|_{L^2}^\theta\|Y^n-Y\|_{L^q}^{1-\theta}.
  \end{equation*}
  Since $\|Y^n-Y\|_{L^2}\to 0$ and $\sup_{n\ge 1}\|Y^n-Y\|_{L^q}<\infty$, we have
  \begin{equation*}
    \E\Bigl[\sup_{t\in[0,T]}|Y_t^n-Y_t|^p\Bigr]\to 0.
  \end{equation*}
  Consequently, for the given $\epsilon>0$, we may choose $n_\epsilon\in\N$ sufficiently large such that
  \begin{equation*}
    \E\Bigl[ \sup_{t\in[0,T]} |Y_t^{n_\epsilon}-Y_t|^p \Bigr] \leq \frac{\epsilon}{2^pT}.
  \end{equation*}
  We now set
  \begin{equation*}
    Y^\epsilon:=Y^{n_\epsilon}, \qquad \mu^\epsilon:=\mu^{n_\epsilon}, \qquad \sigma^\epsilon:=\sigma^{n_\epsilon}.
  \end{equation*}
  Then, $Y^\epsilon$ solves the smooth SDE
  \begin{equation*}
    Y_t^\epsilon= y_0+\int_0^t \mu^{\epsilon}(s,Y_s^{\epsilon})\dd s +\int_0^t \sigma^{\epsilon}(s,Y_s^{\epsilon})\dd W_s, \qquad t\in[0,T].
  \end{equation*}
  Hence,
  \begin{equation*}
    \E\Bigl[\sup_{t\in[0,T]} |Y_t^\epsilon-Y_t|^p \Bigr]\leq \frac{\epsilon}{2^pT}.
  \end{equation*}
  
  Using the uniform $L^p$-boundedness of $Y$, we deduce
  \begin{equation}\label{eq: smooth SDE bounded}
    \E\Bigl[\sup_{t\in[0,T]} |Y_t^\varepsilon|^p\Bigr]
      \le 2^{p-1}\Bigl( \E\Bigl[\sup_{t\in[0,T]} |Y_t^\varepsilon - Y_t|^p\Bigr]+ \E\Bigl[\sup_{t\in[0,T]} |Y_t|^p\Bigr]\Bigr)
      < \infty.
  \end{equation}

  \textit{Step~2.} We next rewrite $Y_t^\epsilon$ as the solution of a Stratonovich SDE. Using the usual It\^o--Stratonovich correction, we can write
  \begin{align*}
    \dd Y_t^\epsilon&=\mu^\epsilon(t,Y_t^\epsilon)\dd t+ \sigma^\epsilon(t,Y_t^\epsilon)\dd W_t\\
    &=(\mu^\epsilon(t,Y_t^\epsilon) -\frac{1}{2}\sigma^\epsilon(t,Y_t^\epsilon)\frac{\partial \sigma^\epsilon}{\partial y}(t,Y_t^\epsilon))\dd t+\sigma^\epsilon(t,Y_t^\epsilon)\circ\dd W_t\\
    &=\tilde\mu^\epsilon(t,Y_t^\epsilon)\dd t+\sigma^\epsilon(t,Y_t^\epsilon)\circ\dd W_t,
  \end{align*}
  where $\circ$ denotes Stratonovich integration and $\tilde\mu^\epsilon$ is a modification of $\mu^\epsilon$ by the additional drift term. Introducing the time-extended Brownian motion $\widehat W_t = (t,W_t)$, we may rewrite the SDE in the compact Stratonovich form
  \begin{equation}\label{eq: smooth Strato SDE}
    \d Y^\epsilon_t=\widehat{\sigma}^\epsilon(t,Y^\epsilon_t)\circ\dd\widehat{W}_t,
  \end{equation}
  where $\widehat{\sigma}^\epsilon$ now also contains the drift term $\tilde{\mu}^\epsilon$, i.e., $\widehat{\sigma}^\epsilon\colon [0,T]\times\R^m\to\R^{m\times(d+1)}$ with
  \begin{equation*}
    \widehat{\sigma}^\epsilon=
    \begin{pmatrix}
      \tilde{\mu}^\epsilon_1 & \sigma^\epsilon_{11}& \cdots & \sigma^\epsilon_{1d} \\
      \tilde{\mu}^\epsilon_2 & \sigma^\epsilon_{21} & \cdots & \sigma^\epsilon_{2d} \\
      \vdots  & \vdots  & \ddots & \vdots  \\
      \tilde{\mu}^\epsilon_{m} & \sigma^\epsilon_{m1} & \cdots & \sigma^\epsilon_{md}
    \end{pmatrix}.
  \end{equation*}
  By construction we have $\widehat\sigma^\epsilon\in C_b^3([0,T]\times \R^m;\mathcal L(\R^{d+1},\R^m))$. Hence, by \cite[Theorem~8.3]{Friz2020}, the associated rough differential equation (RDE), given by
  \begin{equation}\label{eq: RDE}
    \d Y_t^\epsilon=\widehat\sigma^\epsilon(t,Y_t^\epsilon)\dd\hbW_t,
  \end{equation}
  driven by the time-extended Brownian rough path $\hbW$, is well-posed and admits a unique global solution.

  Moreover, by \cite[Theorem~9.1]{Friz2020}, \eqref{eq: smooth Strato SDE} can be solved pathwise almost surely as a RDE solution $(Y_t^\epsilon(\omega),\widehat\sigma^\epsilon(t,Y_t^\epsilon(\omega)))\in\mathcal D^{2\alpha}_{W(\omega)}$ of \eqref{eq: RDE}.
    
  \textit{Step~3.} Let $\Phi\colon \Lambda_T^\alpha \to \R^m$ denote the solution map to \eqref{eq: RDE},  i.e.~$\Phi(\hbW_{[0,t]}) = Y_t^\epsilon.$ Then,
  \begin{align*}
    \int_{\Lambda_T^\alpha}|\Phi|^p\dd\nu&=\int_{\widehat C_{d,T}^\alpha}\int_0^T|(\Phi\circ\phi)(t,\hbW)|^p\dd t\dd\mu_{\hbW}\\
    &=\E\Bigl[\int_0^T|\Phi(\hbW_{[0,t]})|^p\dd t\Bigr]\\
    &=\E\Bigl[\int_0^T|Y_t^\epsilon|^p\dd t\Bigr]\\
    &\le T \E\Bigl[\sup_{t\in[0,T]}|Y_t^\epsilon|^p\Bigr]<\infty,
  \end{align*}
  where we used a change of measure result and \eqref{eq: smooth SDE bounded}. Thus, $\Phi\in L^p(\Lambda_T^\alpha;\R^m)$ and we may apply Theorem~\ref{thm:Lpmain 2} componentwise, since the integrability condition is in particular satisfied by the time-extended Brownian rough path, see Corollary~\ref{cor: fBM}. The extension from the scalar-valued to the $\R^m$-valued case is immediate, so we omit the details. Therefore, for every $\epsilon>0$ there exists a functional $f_\ell\in\mathcal L_\Lambda$, such that
  \begin{equation*}
    \|\Phi-f_\ell\|^p_{L^p(\Lambda_T^\alpha)}<\frac{\epsilon}{2^p}.
  \end{equation*}
  This yields that there exists a linear function $\boldsymbol\ell$ on the Brownian signature, such that
  \begin{align*}
    \E\Bigl[\int_0^T|Y_t^\epsilon-\boldsymbol\ell(\hbbW_t)|^p\dd t\Bigr]&=\E\Bigl[\int_0^T|\Phi(\hbW_{[0,t]})-f_\ell(\hbW_{[0,t]})|^p\dd t\Bigr]\\
    &=\int_{\widehat C_{d,T}^\alpha}\int_0^T|(\Phi\circ\phi-f_\ell\circ\phi)(t,\hbW)|^p\dd t\dd\mu_{\hbW}\\
    &=\int_{\Lambda_T^\alpha}|\Phi-f_\ell|^p\dd\nu\\
    &=\|\Phi-f_\ell\|^p_{L^p(\Lambda_T^\alpha)}<\frac{\epsilon}{2^p},
  \end{align*}
  where we used a change of measure result. Finally, combining steps $1$-$3$ and using the triangle inequality, we obtain
  \begin{align*}
    &\E\Bigl[\int_0^T|Y_t-\boldsymbol\ell(\hbbW_t)|^p\dd t\Bigr]\\
    &\le 2^{p-1}\Bigl(\E\Bigl[\int_0^T|Y_t-Y_t^\epsilon|^p\dd t\Bigr]+\E\Bigl[\int_0^T|Y_t^{\epsilon}-\boldsymbol\ell(\hbbW_t)|^p\dd t\Bigr]\Bigr)\\
    &\le 2^{p-1}\Bigl( T\E \Bigl[\sup_{t\in[0,T]}|Y_t-Y_t^\epsilon|^p\Bigr]+\E\Bigl[\int_0^T|Y_t^{\epsilon}-\boldsymbol\ell(\hbbW_t)|^p\dd t\Bigr]\Bigr)\\
    &<T\frac{\epsilon}{2T}+\frac \epsilon 2\\
    &<\epsilon,
  \end{align*}
  which yields the desired result.
\end{proof}

\begin{remark}
  Proposition~\ref{prop: SDE signature approx} can alternatively be proved by a direct application of Proposition~\ref{prop: approximation Gaussian signature}~\emph{(ii)}. Indeed, on the canonical Wiener space, any $(\mathcal F_t^W)$-progressively measurable process $Y\in\mathcal H^p$ (in particular, strong solutions of It{\^o} SDEs under standard assumptions on the coefficients) can be written in the form
  \begin{equation*}
    Y_t = f\big(\widehat W_{[0,t]}),\quad t\in [0,T],
  \end{equation*}
  for some non-anticipative functional $f$, where $\widehat W_t = (t,W_t)$ denotes the time-extended Brownian motion, cf. \cite[Chapter~5.3.D]{Karatzas1988}. If $\hbW$ is the time-extended Stratonovich-enhanced Brownian rough path and $\pi_1$ its first-level projection, then $\widehat W_{[0,t]} = \pi_1(\hbW_{[0,t]})$, and thus
  \begin{equation*}
    Y_t = f(\widehat W_{[0,t]})
    = f(\pi_1(\hbW_{[0,t]}))
    =: \Phi(\hbW_{[0,t]}).
  \end{equation*}
  Hence, $Y$ fits into the setting of Proposition~\ref{prop: approximation Gaussian signature}~\emph{(ii)}, which then yields an $\mathcal H^p$-approximation of $Y$ by linear functionals on the time-extended Brownian signature.

  We note, however, that making the representation $Y_t = f(\widehat W_{[0,t]})$ fully rigorous requires a careful measurability analysis for progressively measurable processes with respect to the topology induced by the rough path distance used on $\Lambda_T^\alpha$; cf.~\cite[Section~4.2]{Bank2025}. For this reason, we have opted for the proof of Proposition~\ref{prop: SDE signature approx} based on classical results from the theory of stochastic differential equations and rough paths.
\end{remark}

\begin{remark}
  Recently, so-called signature-based models have been introduced in mathematical finance in \cite{Cuchiero2023,Cuchiero2025}; see also \cite{Arribas2021}. These models offer several favorable features compared to classical approaches, which are typically based on stochastic differential equations, for describing financial markets. More precisely, signature models represent the underlying dynamics as linear functionals acting on the random signature of a driving noise process, with the time-extended Brownian motion being the most commonly used example. 
  Proposition~\ref{prop: SDE signature approx} provides a theoretical density result for this  class of Brownian signature models in the classical (unweighted) $L^p$-sense. In particular, it shows that finite linear combinations of time-extended Brownian signature coordinates can approximate solutions to a broad class of stochastic differential equations, independently of the specific drift and diffusion structure. This complements the existing numerical literature on signature-based models in finance, where such finite-dimensional signature models have already been successfully applied, for instance, to the calibration of Heston- and SABR-type stochastic volatility models in \cite{Cuchiero2023}.
\end{remark}

\appendix
\section{Auxiliary results on stopped rough paths}\label{sec: appendix}

In this appendix, we collect auxiliary results on stopped rough paths. We prove that the stopped rough path extension introduced in
Definition~\ref{def: stopped rough path} is well-defined, in the sense that the stopped lifts of smooth time-extended approximations converge and the limit is independent of the chosen approximation. We further show that the metric topology on the space of stopped rough paths coincides with the final topology induced by the canonical restriction map.

\begin{lemma}\label{lem: estimate appendix}
  Let $\alpha\in(0,1]$ and $N\in\N$. Further, let $G=G^N(\R^{d+1})$ and let $\tau_a:=\exp(ae_0)$. For every $T>0$ there exists a constant $C_T>0$ such that, for all $a,b\ge0$ with $a+b\le T$, all $\delta\in[0,1]$, and all $z\in G$ satisfying
  \begin{equation*}
    \|z\|_{cc}\le \delta b^\alpha ,
  \end{equation*}
  one has
  \begin{equation*}
    \|\tau_a^{-1}\otimes z\otimes\tau_a\|_{cc}\le C_T\delta^{1/N}(a+b)^\alpha .
  \end{equation*}
\end{lemma}

\begin{proof}
  Write $z=\exp(\xi)$, with $\xi\in\mathfrak g^N(\R^{d+1})$, i.e.
  \begin{equation*}
    \xi=(0,\xi^{(1)},\ldots,\xi^{(N)}),
    \qquad
    \xi^{(k)}\in(\R^{d+1})^{\otimes k}.
  \end{equation*}
  For $h$ in the Lie algebra, the adjoint map is defined by
  \begin{equation*}
    \operatorname{ad}{e_0}(h):=[e_0,h].
  \end{equation*}
  By the adjoint identity \cite[Lemma~7.22]{Friz2010}, applied with $-ae_0$, we get
  \begin{equation*}
    \tau_a^{-1}\otimes z\otimes \tau_a=\exp(-ae_0)\otimes \exp(\xi)\otimes \exp(ae_0)=\exp\big(e^{-a\operatorname{ad}{e_0}}\xi\big).
  \end{equation*}
  Equivalently,
  \begin{equation*}
    \log(\tau_a^{-1}\otimes z\otimes \tau_a)=e^{-a\operatorname{ad}{e_0}}\xi .
  \end{equation*}
  Using the Baker--Campbell--Hausdorff formula \cite[Theorem~7.24]{Friz2010}, this exponential of the adjoint is a finite sum of iterated brackets:
  \begin{equation*}
    e^{-a\operatorname{ad}{e_0}}\xi=\sum_{i=0}^{N-1}\frac{(-a)^\ell}{i!}(\operatorname{ad}{e_0})^{i}\xi .
  \end{equation*}
  Since $\operatorname{ad}{e_0}$ increases the homogeneous degree by one, the $k$-th component $\eta^{(k)}$ of $\eta:=\log(\tau_a^{-1}\otimes z\otimes \tau_a)$ satisfies
  \begin{equation*}
    |\eta^{(k)}|\le C_T \sum_{j=1}^k a^{k-j}|\xi^{(j)}| .
  \end{equation*}
  By the equivalence of homogeneous norms, see \cite[Theorem~7.44]{Friz2010}, the homogeneous norm $g\mapsto \max_{1\le j\le N}|(\log g)^{(j)}|^{1/j}$ is equivalent to the Carnot--Carathéodory norm. Hence, the assumption $\|z\|_{cc}\le \delta b^\alpha$ implies
  \begin{equation*}
    |\xi^{(j)}|\le C\|z\|_{cc}^j\le C\delta^j b^{\alpha j}, \qquad j=1,\dots,N.
  \end{equation*}
  Since $a+b\le T$, $\alpha\le1$, $k\le N$, and $\delta\in[0,1]$, we obtain
  \begin{equation*}
    |\eta^{(k)}|\le C_T\delta^{k/N}(a+b)^{\alpha k}.
  \end{equation*}
  Using again the equivalence of homogeneous norms yields
  \begin{equation*}
    \|\tau_a^{-1}\otimes z\otimes\tau_a\|_{cc}
    =\|\exp(\eta)\|_{cc}
    \le C_T\max_{1\le k\le N}|\eta^{(k)}|^{1/k}
    \le C_T\delta^{1/N}(a+b)^\alpha .
  \end{equation*}
\end{proof}

\begin{lemma}
  Let $\alpha\in(0,1]$, let $N:=\lfloor 1/\alpha\rfloor$, and let $0\le t\le T$. Let
  \begin{equation*}
    \hbX_{[0,t]}
    \in C^{0,\alpha}([0,t];G^N(\R^{d+1}))
  \end{equation*}
  be a time-extended geometric rough path. Suppose that $\hX^n_r=(r,X^n_r)$, $r\in[0,t]$, is a sequence of smooth time-extended paths whose canonical lifts $\hbX^n$ satisfy
  \begin{equation*}
    d_{cc,\alpha;[0,t]} (\hbX^n,\hbX)\to 0 .
  \end{equation*}
  For $r\in[0,T]$, define the stopped paths $\hX^{n,t}_r:=(r,X^n_{r\wedge t})$, and denote their canonical lifts on $[0,T]$ by $\hbX^{n,t}$. Then, the sequence $(\hbX^{n,t})_{n\in\N}$ converges in $d_{cc,\alpha;[0,T]}$. Moreover, the resulting limit is independent of the approximating sequence $(\hX^n)_{n\in\N}$. Consequently, the stopped rough path
  \begin{equation*}
    \hbX^{t}_{[0,T]}:=\lim_{n\to\infty}\hbX^{n,t}_{[0,T]}
  \end{equation*}
  is well-defined.
\end{lemma}

\begin{proof}
  Let $G:=G^N(\R^{d+1})$, where $N=\lfloor 1/\alpha\rfloor$. For $a\ge0$, write
  \begin{equation*}
    \tau_a:=\exp(ae_0)\in G .
  \end{equation*}
  For the canonical lift of the stopped path $r\mapsto (r,X^n_{r\wedge t})$, Chen's relation yields, for $0\le u\le v\le T$,
  \begin{equation*}
    \hbX^{n,t}_{u,v}
    =
    \begin{cases}
    \hbX^{n}_{u,v},
    & 0\le u\le v\le t,\\[3pt]
    \tau_{v-u},
    & t\le u\le v\le T,\\[3pt]
    \hbX^{n}_{u,t}\otimes \tau_{v-t},
    & 0\le u<t<v\le T .
    \end{cases}
  \end{equation*}

  Set
  \begin{equation*}
    \delta_{n,m}
    := d_{cc,\alpha;[0,t]}
    (\hbX^{n},\hbX^{m} ).
  \end{equation*}
  Since $\hbX^n\to\hbX$ in $d_{cc,\alpha;[0,t]}$, we have $\delta_{n,m}\to0$ as $n,m\to\infty$. In particular, for all sufficiently large $n,m$, we may assume that $\delta_{n,m}\le1$.

  Let $0\le u<v\le T$. If $0\le u\le v\le t$, then
  \begin{equation*}
    d_{cc}(\hbX^{n,t}_{u,v},\hbX^{m,t}_{u,v})= d_{cc}
    (\hbX^{n}_{u,v},\hbX^{m}_{u,v} )
    \le \delta_{n,m}|v-u|^\alpha .
  \end{equation*}
  If $t\le u\le v\le T$, both increments are equal to $\tau_{v-u}$, and the distance is zero.

  It remains to consider $u<t<v$. Put
  \begin{equation*}
    a:=v-t, \qquad  b:=t-u .
  \end{equation*}
  Then $a+b=v-u$, and by left-invariance of $d_{cc}$,
  \begin{equation*}
    \begin{aligned}
    &d_{cc}(\hbX^{n,t}_{u,v},\hbX^{m,t}_{u,v}) \\
    &\quad =\|\tau_a^{-1}\otimes ((\hbX^{n}_{u,t})^{-1}\otimes \hbX^{m}_{u,t})\otimes\tau_a \|_{cc} .
    \end{aligned}
  \end{equation*}
  With $z:=(\hbX^{n}_{u,t})^{-1}\otimes \hbX^{m}_{u,t},$ we have $\|z\|_{cc}\le \delta_{n,m} b^\alpha.$
  Hence, by Lemma \ref{lem: estimate appendix}, we have
  \begin{equation*}
    d_{cc}(\hbX^{n,t}_{u,v},\hbX^{m,t}_{u,v})
    \le C_T\delta_{n,m}^{1/N}(a+b)^\alpha
    = C_T\delta_{n,m}^{1/N}|v-u|^\alpha .
  \end{equation*}
  Combining the three cases and taking the supremum over $0\le u<v\le T$ gives, for all sufficiently large $n,m$,
  \begin{equation*}
    d_{cc,\alpha}(\hbX^{n,t},\hbX^{m,t})
    \le \max\Bigl\{\delta_{n,m},C_T\delta_{n,m}^{1/N}\Bigr\}.
  \end{equation*}
  Since $\delta_{n,m}\to0$, the right-hand side tends to zero. Hence, $(\hbX^{n,t})_n$ is a Cauchy sequence in $d_{cc,\alpha}$. By completeness of the geometric $\alpha$-H{\"o}lder rough path space, this sequence converges.

  It remains to show that the limit is independent of the chosen smooth approximation. Let $\widehat Y^m_r=(r,Y^m_r)$, $r\in[0,t]$, be another sequence of smooth time-extended paths whose canonical lifts $\widehat{\mathbf Y}^m$ converge to $\hbX$ in $d_{cc,\alpha;[0,t]}$. Let $\widehat{\mathbf Y}^{m,t}$ denote the canonical lifts of the stopped paths $r\mapsto (r,Y^m_{r\wedge t})$ on $[0,T]$.

  For sufficiently large $n,m$, the estimate above applied to $\hbX^n$ and $\hbY^m$ gives
  \begin{equation*}
    d_{cc,\alpha}
    (\hbX^{n,t},\hbY^{m,t})\le \max\Bigl\{d_{cc,\alpha;[0,t]}(\hbX^{n},\hbY^{m}),
    C_T\Bigl(d_{cc,\alpha;[0,t]}(\hbX^{n},\widehat{\mathbf Y}^{m})\Bigr)^{1/N}\Bigr\}.
  \end{equation*}
  Furthermore,
  \begin{equation*}
    \begin{aligned}
    d_{cc,\alpha;[0,t]}
    (\hbX^{n},\widehat{\mathbf Y}^{m} )
    &\le  d_{cc,\alpha;[0,t]}( \hbX^{n},
    \hbX) \\
    &\quad+ d_{cc,\alpha;[0,t]}( \hbX, \widehat{\mathbf Y}^{m}),
    \end{aligned}
  \end{equation*}
  and the right-hand side tends to zero as $n,m\to\infty$. Thus the two stopped approximating sequences have the same limit. Finally, since $\hbX^{n,t}_{[0,t]}=\hbX^n,$ passing to the limit gives $\hbX^{t}_{[0,t]}=\hbX$. Hence
  \begin{equation*}
    \hbX^{t}_{[0,T]} := \lim_{n\to\infty}\hbX^{n,t}_{[0,T]}
  \end{equation*}
  is well-defined, independent of the chosen approximating sequence, and extends $\hbX$ on $[0,t]$.
\end{proof}

\begin{lemma}\label{lem: phi cts}
  The function
  \begin{align*}
    \phi\colon[0,T]\times \widehat{C}_{d,T}^\alpha&\to \Lambda_T^\alpha,\\
    (t,\hbX)&\mapsto \hbX_{[0,t]},
  \end{align*}
  is continuous.
\end{lemma}

\begin{proof}
  Recall that $\widehat{C}_{d,T}^\alpha$ is endowed with the topology induced by $d_{cc,\alpha^\prime}$ for $\alpha^\prime<\alpha$. Let $(s,\hbY)\in[0,T]\times \widehat C_{d,T}^{\alpha}$ be fixed. We show continuity of $\phi$ at $(s,\hbY)$. For $(t,\hbX)\in[0,T]\times \widehat C_{d,T}^{\alpha}$, we have
  \begin{align*}
    d_{\Lambda,\alpha^\prime}(\phi(t,\hbX),\phi(s,\hbY))&=d_{\Lambda,\alpha^\prime}(\hbX_{[0,t]},\hbY_{[0,s]})\\
    &=|t-s|+d_{cc,\alpha^\prime}(\hbX^t_{[0,T]},\hbY^s_{[0,T]})\\
    &\le |t-s|+d_{cc,\alpha^\prime}(\hbX^t_{[0,T]},\hbY^t_{[0,T]})+d_{cc,\alpha^\prime}(\hbY^t_{[0,T]},\hbY^s_{[0,T]}).
  \end{align*}
  Now let $(t,\hbX)\to(s,\hbY)$, that is,
  \begin{equation*}
    |t-s|+d_{cc,\alpha^\prime}(\hbX,\hbY)\to0.
  \end{equation*}
  Then, we immediately obtain that
  \begin{equation*}
    |t-s|\to 0,\quad d_{cc,\alpha^\prime}(\hbX^t_{[0,T]},\hbY^t_{[0,T]})\to 0,\text{ and } d_{cc,\alpha^\prime}(\hbY^t_{[0,T]},\hbY^s_{[0,T]})\to 0.
  \end{equation*}
\end{proof}

\begin{lemma}\label{lem: topologies}
  The topology on $\Lambda_T^\alpha$ coincides with the topology induced by $\phi\colon [0,T]\times \widehat{C}_{d,T}^\alpha\to \Lambda_T^\alpha$ and $\Lambda_T^\alpha$ equipped with $d_{\Lambda,\alpha}$ is Polish.
\end{lemma}

\begin{proof}
  A set $U\subset \Lambda_T^\alpha$ is open with respect to the final topology if and only if $\phi^{-1}(U)$ is open in $[0,T]\times \widehat C_{d,T}^{\alpha}$, where $\widehat C_{d,T}^{\alpha}$ is equipped with the topology induced by $d_{cc,\alpha^\prime;[0,T]}$. By the previous lemma, $\phi$ is continuous for the topology induced by $d_{\Lambda,\alpha^\prime}$. Hence, if $U\subset \Lambda_T^\alpha$ is $d_{\Lambda,\alpha^\prime}$-open, then $\phi^{-1}(U)$ is open. Thus every $d_{\Lambda,\alpha^\prime}$-open set is open with respect to the final topology.

  Conversely, assume that $\phi^{-1}(U)$ is open for some set $U\subset \Lambda_T^\alpha$. Let $\hbX_{[0,t]}\in U$. We show that $\hbY_{[0,s]}\in U$ whenever $d_{\Lambda,\alpha'}(\hbX_{[0,t]},\hbY_{[0,s]})$ is sufficiently small. Since the stopped extension $\hbX^t_{[0,T]}$ belongs to $\widehat C_{d,T}^{\alpha}$ and satisfies
  \begin{equation*}
    \phi(t,\hbX^t)=\hbX_{[0,t]},
  \end{equation*}
  we have $(t,\hbX^t)\in\phi^{-1}(U)$. Since $\phi^{-1}(U)$ is open in $[0,T]\times \widehat C_{d,T}^{\alpha}$ and $(t,\hbX^t)\in\phi^{-1}(U)$, there exists $\epsilon>0$ such that, for every $(r,\widehat{\mathbf Z})\in[0,T]\times \widehat C_{d,T}^{\alpha}$,
  \begin{equation*}
    |t-r|+d_{cc,\alpha^\prime}(\hbX^t_{[0,T]},\widehat{\mathbf Z}_{[0,T]})<\epsilon
  \end{equation*}
  implies $(r,\widehat{\mathbf Z})\in\phi^{-1}(U).$ Now let $\hbY_{[0,s]}\in\Lambda_T^\alpha$ satisfy
  \begin{equation*}
    d_{\Lambda,\alpha'}(\hbX_{[0,t]},\hbY_{[0,s]})<\epsilon.
  \end{equation*}
  Choosing $r=s$ and $\widehat{\mathbf Z}=\hbY^s$. Then,
  \begin{equation*}
    |t-s|+d_{cc,\alpha^\prime}(\hbX^t_{[0,T]},\hbY^s_{[0,T]})=d_{\Lambda,\alpha^\prime}(\hbX_{[0,t]},\hbY_{[0,s]})<\epsilon.
  \end{equation*}
  Hence $(s,\hbY^s)\in\phi^{-1}(U)$. Applying $\phi$, we obtain
  \begin{equation*}
    \phi(s,\hbY^s)=\hbY_{[0,s]}\in U.
  \end{equation*}
  Moreover, $\Lambda_T^\alpha$, equipped with $d_{\Lambda,\alpha}$, is Polish. Indeed, separability follows from the separability of $[0,T]\times\widehat C_{d,T}^\alpha$, while completeness follows from the completeness of $[0,T]$ and $\widehat C_{d,T}^\alpha$, together with the definition of $d_{\Lambda,\alpha}$.
\end{proof}

\bibliography{quellen}{}
\bibliographystyle{amsalpha}

\end{document}